\documentclass[a4paper,11pt]{amsart}

\usepackage[utf8]{inputenc}		
\usepackage[T1]{fontenc}
\usepackage[english]{babel}

\usepackage{amsfonts}			
\usepackage{amsmath}
\usepackage{amssymb}
\usepackage{amsthm}

\usepackage{graphicx}			
\usepackage{tikz}

\usepackage{hyperref}			
	\hypersetup{colorlinks=false, urlcolor=black, linkcolor=black}

\usepackage{enumitem}			

\usepackage{color}				

\newcommand{\dtext}{\textnormal d}
\usepackage{mathrsfs}
\newcommand{\onto}{\xrightarrow[]{{}_{\!\!\textnormal{onto\,\,}\!\!}}}

\newcommand{\R}{\mathbb{R}}						
\newcommand{\C}{\mathbb{C}}						

\renewcommand{\S}{\mathbb{S}}					
\newcommand{\D}{\mathbb{D}}						

\newcommand{\X}{\mathbb{X}}						
\newcommand{\Y}{\mathbb{Y}}

\newcommand{\eps}{\varepsilon}					

\newcommand{\dd}								
	{\mathop{}\!\mathrm{d}}						
\newcommand{\ddn}[1]							
	{\mathop{}\!\mathrm{d^{#1}}}

\newcommand{\abs}[1]							
	{\left| #1 \right|}
\newcommand{\smallabs}[1]						
	{\lvert #1 \rvert}	
\newcommand{\norm}[1]							
	{\left\lVert #1 \right\rVert}	
\newcommand{\smallnorm}[1]						
	{\lVert #1 \rVert}						
\newcommand{\ip}[2]								
	{\left< #1 , #2 \right>}

\DeclareMathOperator{\id}{id}					
				
\DeclareMathOperator{\capac}{Cap}			

\DeclareMathOperator{\vmo}{VMO}
\DeclareMathOperator{\sym}{Sym}

\let\Re\relax									
\let\Im\relax
\DeclareMathOperator{\Re}{Re}
\DeclareMathOperator{\Im}{Im}					

\newcommand{\loc}{\mathrm{loc}}
\DeclareMathOperator{\osc}{osc}

\newtheorem{thm}{Theorem}[section]{\bf}{\it}
\newtheorem{lemma}[thm]{Lemma}
\newtheorem{prop}[thm]{Proposition}
\newtheorem{cor}[thm]{Corollary}

\theoremstyle{definition}
\newtheorem{defn}[thm]{Definition}

\newtheorem{example}[thm]{Example}

\theoremstyle{remark}
\newtheorem{rem}[thm]{Remark}

\numberwithin{equation}{section}

\begin{document}

\title[Analytic characterization  of  monotone Hopf-harmonics]{Analytic characterization  of\\ monotone Hopf-harmonics}

\author[I. Kangasniemi]{Ilmari Kangasniemi}
\address{Department of Mathematics, Syracuse University, Syracuse,
NY 13244, USA }
\email{kikangas@syr.edu}

\author[A. Koski]{Aleksis Koski}
\address{Department of Mathematics and Statistics, P.O. Box 68 (Pietari Kalmin katu 5), FI-00014 University of Helsinki, Finland}
\email{aleksis.koski@helsinki.fi}

\author[J. Onninen]{Jani Onninen}
\address{Department of Mathematics, Syracuse University, Syracuse,
NY 13244, USA and  Department of Mathematics and Statistics, P.O.Box 35 (MaD) FI-40014 University of Jyv\"askyl\"a, Finland
}
\email{jkonnine@syr.edu}
\thanks{A. Koski was supported by the Academy of Finland Grant number 307023.
J. Onninen was supported by the NSF grant  DMS-1700274.}

\subjclass[2020]{Primary 31C45; Secondary 35J25, 58E20, 74B20, 46E35}


\keywords{Hopf-Laplace equation, holomorphic quadratic differentials, inner-variational equations,  monotone mappings, orientation-preserving Sobolev mappings,  the principle of non-interpenetration of matter. }

\begin{abstract} 
We study solutions of the inner-variational equation associated with the Dirichlet energy in the plane, given homeomorphic Sobolev boundary data.
We prove that such a solution is  monotone if and only if its Jacobian determinant  does not change sign. These solutions, called {\it monotone Hopf-harmonics}, are a natural alternative to harmonic homeomorphisms.  Examining the topological behavior of a solution (not a priori monotone) on the trajectories of Hopf quadratic  differentials plays a sizable role in our arguments.
\end{abstract}

\maketitle

\section{Introduction}
A fundamental problem in the theory of Nonlinear Elasticity (NE)~\cite{Anb, Bac, Cib} and Geometric Function Theory (GFT)~\cite{AIMb, HKb, IMb, Reb, Rib} is to determine what analytic data on a given mapping $h \colon \X \to \Y$ provides topological information on the deformation that $h$ represents. An exemplary result of this type is the theorem of Reshetnyak~\cite{Re} which states that the continuous representative of a quasiregular mapping is either constant or both open and discrete. His remarkable theorem hence  gives topological conclusions from analytic assumptions.  

Along the same lines, we study under what conditions an inner variatiational  minimizer of the Dirichlet energy is monotone,
given homeomorphic boundary data. An \emph{inner variation} of a map $h$ is a change of independent variable $h_\varepsilon(x) = h(x + \varepsilon \eta)$, where $\eta \in  C^\infty_0 (\X, \mathbb C)$. On the other hand, a continuous map $h \colon \X \to \C$ is \emph{monotone} if $h^{-1}\{y\}$ is connected for every $\{y\} \in \C$. The definition of monotone maps is due to Morrey~\cite{Mor}, and is purely topological in nature. 

The inner variational minimizers of Dirichlet energy satisfy the \emph{Hopf-Laplace equation}, named in recognition of Hopf's work~\cite{Ho}. The Hopf-Laplace equation is a second-order partial differential equation for maps $h \colon \X \to \C$ on a planar domain $\X \subset \C$, given by
\begin{equation}\label{eq:hopf-harmonic}
	\frac{\partial}{ \partial \bar z} \left(h_z \, \overline{h_{\bar z}}\right)=0 \, ,  \qquad  h\in  W^{1,2}_{\loc}({\X}, \mathbb C).
\end{equation}
The Hopf-Laplace equation is also known as the \textit{energy-momentum} or \textit{equilibrium} equation, etc.~\cite{Cob, SSe, Ta}.  In continuum mechanics the inner variation is often called a {\it domain variation}~\cite{GS, Heb, HPb, KSb}. 

A complex-valued harmonic mapping $h=u+iv$ always solves the Hopf-Laplace equation. Conversely, homeomorphic solutions to the Hopf-Laplace equation~\eqref{eq:hopf-harmonic} are harmonic~\cite{IKOhopf}. Arbitrary solutions, however, may behave surprisingly wildly. Indeed, Iwaniec, Verchota and Vogel~\cite{IVV} constructed a nonconstant piecewise orthogonal mapping $ h \colon  \overline{\X} \to \mathbb C$ vanishing on $\partial \X$  whose Hopf product $h_z \overline{h_{\bar z}}$ is identically zero. Such wild solutions are often undesirable in NE and GFT; hence, applications of the Hopf-Laplace equation generally consider only continuous and monotone solutions of \eqref{eq:hopf-harmonic}, which are  called \emph{monotone Hopf-harmonics}. These solutions only allow for \emph{weak interpenetration of matter}; roughly speaking, squeezing of a portion of the material to a point may occur, but not folding or tearing.

In NE and GFT, an axiomatic assumption for the deformations $h \colon \X \to \mathbb C$ studied is that 
\begin{equation}\label{eq:jacob}
	J_h(x)=\det Dh(x) \ge 0 \qquad  \textnormal{ almost everywhere (a.e.)}
\end{equation}
Particularly in elasticity theory, this constraint is of utmost importance, since it implies that no volume of material is turned ``inside out''; an undesirable property which we refer to as \emph{strong interpenetration of matter}. Our main result tells that, if $h$ is a solution of the Hopf-Laplace equation with homeomorphic boundary values in the Sobolev sense, then the monotonicity of $h$ can be characterized purely by~\eqref{eq:jacob}.

\begin{thm}\label{thm:main}
	Let $\X, \Y \subset \C$ be (simply connected) Jordan domains with $\Y$ Lipschitz, and let $f \colon \overline{\X} \onto \overline{\Y}$ be an orientation-preserving homeomorphism in $W^{1,2} (\X, \C)$. Suppose that a mapping $h \colon \overline{\X} \to \C$ satisfies the Hopf-Laplace equation~\eqref{eq:hopf-harmonic} and that $h\in f+W^{1,2}_0(\X, \C)$. Then the following conditions are equivalent:
	\begin{enumerate}
		\item\label{enum:main_thm_jacobian} $J_h(x) \ge 0$ a.e. in $\X$;
		\item\label{enum:main_thm_monotone} $h$ is continuous up to the boundary of $\X$, $h$ maps $\overline{\X}$ onto $\overline{\Y}$, and the resulting mapping $h \colon \overline{\X} \onto \overline{\Y}$ is monotone.
	\end{enumerate}
\end{thm}

First, we point out that under the assumptions of Theorem~\ref{thm:main} there always exists  a monotone Hopf-harmonic $h \colon \overline{\X} \onto \overline{\Y}$ which coincides with the given $f$ on the boundary of $\X$, see~\cite{Iwaniec-Onninen_Hopf-Harmonics}.  Second, the existence of the  assumed boundary homeomorphism $f \colon \overline{\X} \onto \overline{\Y}$  can be equivalently formulated purely in terms of the boundary map $f \colon \partial \X \onto \partial \Y$, see Section~\ref{sec:traceSobolevmaps}. On the other hand, without the restriction of homeomorphic boundary behavior, the solutions to the Hopf-Laplace equation with nonnegative Jacobian need not be monotone. A simple example of this is the power map $h \colon \overline{\mathbb D} \onto \overline{\mathbb D}$, $h(z)=z^2$, where $J_h \geq 0$ and $h_z \overline{h_{\bar z}} \equiv 0$ but $h$ is not monotone. Furthermore, the assumption that $h$ is a solution of the Hopf-Laplace equation is also essential, as illustrated by the following two examples, the details of which are given in Section \ref{sect:counterexamples}.

\begin{example}\label{ex:1}
	There is a continuous $h \colon \D \to \C$ with $h  \in \id + W^{1,2}_0(\D, \C)$ and $J_h \geq 0$ a.e.\ such that $h(\D) \setminus \overline{\D} \not= \emptyset$ and $h$ fails to be continuous up to the boundary.
\end{example}

\begin{example}\label{ex:2}
	There is a continuous $h \colon \overline{\D} \to \overline{\D}$ with $h \in \id + W^{1,2}_0(\D, \C)$ and $J_h \geq 0$ a.e.\ such that $h^{-1}(y)$ is not connected for a point $y \in \D$.
\end{example}

\subsection{Inner variational problems} A major motivation for the study of inner-variational equations is in energy minimization problems for classes of homeomorphisms or mappings with nonnegative Jacobian, a type of problem common in GFT and NE.  In particular, monotone Hopf-harmonics  become an important resource in both theories in cases where the harmonic extension fails to be injective. This inadequacy in fact occurs for every nonconvex  target domain $\Y$. Indeed, for any such $\Y$ there exists a homeomorphism $f$ from the boundary of the unit disk onto $\partial \Y$ whose harmonic extension fails to be a homeomorphism. Moreover, $f$  takes points in the unit disk $\D$ beyond $\overline{\Y}$ and strong interpenetration of matter hence occurs~\cite{AN, Ch}; see Figure \ref{fig:banaanitarget} for an illustration of this.

\begin{figure}[htbp]
	\centering
	\includegraphics[width=0.45\textwidth]
	{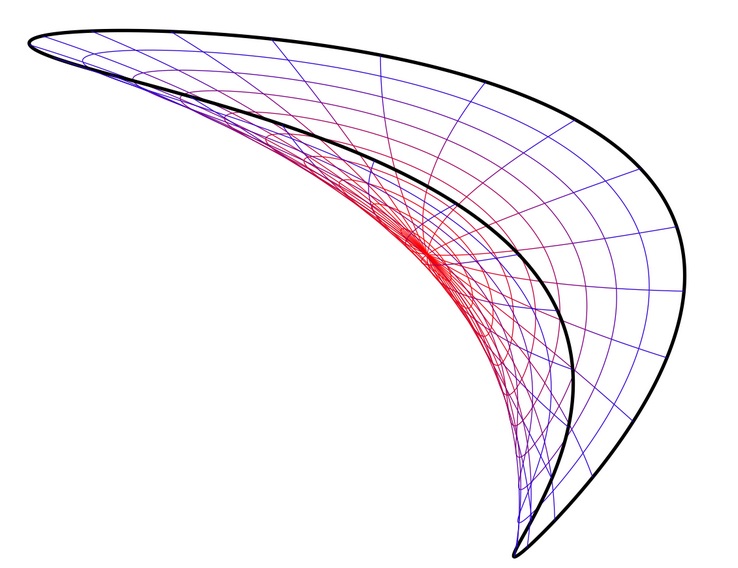}
	\caption{The images of circles and rays of the unit disk $\D$ under the harmonic map $x+iy \mapsto (1-x)^2-y^2+1.6(x+y)+i(x^2-(2+y)^2-x-2y)$ which has homeomorphic boundary data. The folding occurs outside $\overline{ \Y}$.}
	\label{fig:banaanitarget}
\end{figure}

 A  natural approach to avoid such intolerable behavior and obtain a mapping which resembles harmonic homeomorphisms is to minimize the  Dirichlet energy subject to only homeomorphisms. Let $\X, \Y \subset \C$ be Jordan domains with $\Y$ Lipschitz, and $f \colon \overline{ \X }\onto \overline{\Y}$ an orientation-preserving homeomorphism in $W^{1,2} (\X, \mathbb C)$.  Consider  the infimum of the Dirichlet integral
\begin{equation}\label{eq:dirichlet}
	\mathsf E_{_\X}[h] = \int_\mathbb X \abs{Dh(x)}^2 \, \dtext x,
\end{equation}
among the  class of admissible homeomorphism $h$, given by
\[
\begin{split}
	\mathcal H_f^{2} (\overline{\X}, \overline{\Y}) = \{ &h \colon \overline{\X} \onto \overline{\Y} \colon h \textnormal{ a homeomorphism in } W^{1,2}(\X, \C) \textnormal{ s.t. } \\ & h=f \textnormal{ on } \partial \X  \} \,.
\end{split}
\]

The class $\mathcal H_f^{2}(\overline{\X}, \overline{\Y})$ is not closed with respect to this minimization problem. This raises the question of how to properly enlarge the class of  Sobolev  homeomorphisms. In general enlarging the class of admissible deformations can change the nature of the minimization problem. It may result in the infimum of the energy functional changing (the \emph{Lavrentiev gap}) and affect whether or not the infimum is attained.

\subsubsection{The class $\mathcal M_f^{2} (\overline{\X}, \overline{\Y}) $}

The smallest extension of $\mathcal H_f^{2}(\overline{\X}, \overline{\Y})$ that is closed under the minimization problem is given by the class
\[
\begin{split}
	\mathcal M_f^{2} (\overline{\X}, \overline{\Y}) = \{ &h \colon \overline{\D} \onto \overline{\Y} \colon h \textnormal{ a monotone mapping in } W^{1,2}(\X, \C) \textnormal{ s.t. } \\ & h=f \textnormal{ on } \partial \X  \} \, .
\end{split}
\]
of monotone Sobolev maps. Indeed, an infimizing sequence of mappings $h_j \in \mathcal M_f^{2} (\overline{\X}, \overline{\Y})$ converges uniformly and in $W^{1,2}(\X, \C)$ to a map $h \in \mathcal M_f^{2} (\overline{\X}, \overline{\Y})$; see \cite[Remark 3.1]{Iwaniec-Onninen_Hopf-Harmonics}. Moreover, by a Sobolev variant of Youngs' approximation result~\cite{Youngs}, a monotone $h \in \mathcal M_f^{2} (\overline{\X}, \overline{\Y})$ can be approximated with homeomorphisms $h_j \in \mathcal H_f^{2}(\overline{\X}, \overline{\Y})$ strongly in $W^{1,2} (\X, \C)$; see \cite[Theorem 1.3]{IOmono}. Hence, a minimizer of \eqref{eq:dirichlet} exists in $\mathcal M_f^{2} (\overline{\X}, \overline{\Y})$, and there is no {Lavrentiev gap} between the classes $\mathcal M_f^{2} (\overline{\X}, \overline{\Y})$ and $\mathcal H_f^{2}(\overline{\X}, \overline{\Y})$; that is,
\begin{equation}\label{eq:monomin}
	\min_{ h \in \mathcal M_f^{2}(\overline{\X} , \overline{\mathbb Y}) } \mathsf E_{\X}[h]   = \inf_{ h \in \mathcal H_f^{2}(\overline{\X} , \overline{\mathbb Y}) } \mathsf E_{\X}[h] \,  .
\end{equation}

Since one can perform inner variations in the class $\mathcal M_f^{2} (\overline{\X}, \overline{\Y})$, the minimizer $h$ in~\eqref{eq:monomin}  solves the Hopf-Laplace equation \eqref{eq:hopf-harmonic}. Conversely, it was shown in \cite[Proposition 3.4]{Iwaniec-Onninen_Hopf-Harmonics} that a solution of the Hopf-Laplace equation in $\mathcal M_f^{2} (\overline{\X}, \overline{\Y})$ is always an energy minimizer. Thus, the minimizers of the Dirichlet energy in $\mathcal M_f^{2} (\overline{\X}, \overline{\Y})$ are exactly the monotone Hopf-harmonics in $\mathcal M_f^{2} (\overline{\X}, \overline{\Y})$.   
Furthermore, any weak interpenetration of matter under a monotone Hopf-harmonic energy minimizer occurs precisely where the minimizer fails to be harmonic, i.e.\ where it does not  satisfy the Euler-Lagrange equation; see~\cite{IOinv}.

\subsubsection{The class $\mathcal A^2_f (\overline{\X}, \C)$} In mathematical models of NE  one typically allows  the class of competing deformations to be as large as possible. Only physically inappropriate  mappings are excluded, most notably disallowing any strong interpenetration of matter.   This leads us  to  another possible enlargement  of the class $\mathcal H_f^{2}(\overline{\X}, \overline{\Y})$,   given by
\[
\begin{split}
	\mathcal A^2_f (\overline{\X}, \C)= \{ & h \colon \overline{\X} \to \C \colon h\in W^{1,2} (\X, \C), \, J(x,h) \ge 0 \textnormal{ a.e. }  \textnormal{ and } \\
	& h \in  f + W_0^{1,2} (\X, \C)\} \, .
\end{split}
\]
 The class $\mathcal A^2_f (\overline{\X}, \C)$ contains $\mathcal M_f^{2} (\overline{\X}, \overline{\Y})$, but also far more irregular maps. For instance, a mapping in $ \mathcal A^2_f (\overline{\X}, \C) $ need not be continuous in $\X$, and as stated in Examples~\ref{ex:1} and~\ref{ex:2}, such mappings may be non-monotone or may take points in $\X$ beyond $\overline{\Y}$.

The class $\mathcal A^2_f (\overline{\X}, \C)$ has a minimizer $h_\circ$ for the Dirichlet energy \eqref{eq:dirichlet}. This is because $\mathcal A^2_f (\overline{\X} , \C)$ is closed under weak limits in  $W^{1,2} (\X, \C)$,  see \cite[Corollary 8.4.1]{IMb}. Moreover, since $\mathcal A^2_f (\overline{\X} , \C)$ is also closed under inner variations, we again have that the minimizers of \eqref{eq:dirichlet} in $\mathcal A^2_f (\overline{\X}, \C)$ are solutions of the Hopf-Laplace equation in $\mathcal A^2_f (\overline{\X}, \C)$.

This leads to the question of whether or not there is a Lavrentiev gap between $\mathcal A^2_f (\overline{\X}, \C)$ and   the classes $\mathcal M_f^{2} (\overline{\X}, \overline{\Y})$ and $\mathcal H_f^{2}(\overline{\X}, \overline{\Y})$. Indeed, a minimizer in $\mathcal A^2_f (\overline{\X}, \C)$ could a priori not be an element of $\mathcal M_f^{2} (\overline{\X}, \overline{\Y})$. Our Theorem \ref{thm:main} resolves this question, by showing that solutions of the Hopf-Laplace equation in $\mathcal A^2_f (\overline{\X}, \C)$ are in fact elements of $\mathcal M_f^{2} (\overline{\X}, \overline{\Y})$. By combining this with the result \cite[Proposition 3.4]{Iwaniec-Onninen_Hopf-Harmonics} that Hopf-harmonic elements of $\mathcal M_f^{2} (\overline{\X}, \overline{\Y})$ are energy minimizers, the following application is immediate.

\begin{thm}
	Let $\X$ and $\Y$ be Jordan domains with $\Y$ Lipschitz, and let $f \colon \overline{\X} \onto \overline{\Y} $ be an orientation-preserving homeomorphism in  $W^{1,2} (\X, \C)$. Then a mapping $h_\circ \in \mathcal A^2_f (\overline{\X}, \C)$ solves the Hopf-Laplace equation \eqref{eq:hopf-harmonic} if and only if 
	\begin{equation}
		\mathsf E_{\X}[h_\circ]  = \min_{ h \in \mathcal A^2_f (\overline{\X}, \C) }\mathsf E_{\X}[h]  \,.
	\end{equation}
	Moreover, we have
	\begin{equation}
		\min_{ h \in \mathcal A^2_f (\overline{\X}, \C) }\mathsf E_{\X}[h] =\min_{ h \in \mathcal M_f^{2}(\overline{\X} , \overline{\mathbb Y}) } \mathsf E_{\X}[h]   = \inf_{ h \in \mathcal H_f^{2}(\overline{\X} , \overline{\mathbb Y}) } \mathsf E_{\X}[h] \,  .
	\end{equation}
\end{thm}

We remark that the minimizer $h_\circ$ is unique at least when $\Y$ is somewhere convex; see \cite[Theorem 1.8]{Iwaniec-Onninen_Hopf-Harmonics}. Here, a simply connected Jordan domain $\Y\subset \mathbb C\,$ is said to be {\it somewhere convex} if there exists a disk $B^2(y_\circ , \varepsilon)$ with $y_\circ \in \partial \mathbb Y$ and $\varepsilon > 0$ such that $B^2(y_\circ , \varepsilon) \cap \overline{\mathbb Y}$ is convex. Notably, all $\mathscr C^2$-regular domains are somewhere convex. 

The monotone Hopf-harmonic energy-minimizer $h_\circ \colon \overline{\X}\onto \overline{\Y}$ is a harmonic diffeomorphism from $h_\circ^{-1} (\Y)$ onto $\Y$, see~\cite[Section 3.4]{Iwaniec-Onninen_Hopf-Harmonics}.  In particular, the set $\X \setminus h_\circ^{-1} (\Y) $ is squeezed into $\partial \Y$, and no continuum which is compactly contained in $\X$ can be squeezed into a point in $\mathbb Y$. The set $\X \setminus h_\circ^{-1} (\Y)$ may have a positive area. This, for instance, happens for the monotone Hopf-harmonic solution when one chooses the same Jordan domains $\X, \Y$ and boundary map $f$ as in Figure~\ref{fig:banaanitarget}. Note that we always have $J_{h_\circ} \equiv 0$ a.e. in  $\X \setminus h_\circ^{-1} (\Y)$.

Since the landmark paper of Ball~\cite{Bac}, the question of  almost everywhere invertibility of deformations   has been widely studied, for instance,  in the context of Neohookean type energy functionals in NE, see~\cite{Bainv, BP, Cib, Fonseca-Gangbo-book, GMS, MO, MST, Sv, We}. A functional $\mathsf E$ is of \emph{Neohookean type} if 
\[\mathsf E(A) \to \infty \qquad \textnormal{when } \det A \to 0^{+} \, . \]
Such functionals enforce that  all deformations $h$ with finite energy are  \emph{strictly orientation-preserving} in the sense that $J_h(x) > 0$ a.e. The deformations  are  also regular up to the boundary.  Therefore, in the corresponding minimization problems, one can assume that the admissible deformations coincide with a given boundary homeomorphism $f$ in the classical sense. In the case of the Dirichlet energy~\eqref{eq:dirichlet}, however, such restrictions to the admissible deformations in $\mathcal A^2_f (\overline{\X}, \C)$ are not possible. Indeed, even the monotone minimizers of the Dirichlet energy need not be strictly orientation-preserving. Furthermore the condition  $h=f$ on $\partial \X$ is  not preserved under the weak $W^{1,2}$-convergence.  This motivates  the Sobolev trace boundary values assumption in Theorem~\ref{thm:main};  the {\it weak formulation of the Dirichlet problem}.

\subsection{Main ideas of the proof} 
We note that Theorem \ref{thm:main} easily reduces to the case where $\X$ is the unit disk $\D$. Indeed, the Hopf-Laplace equation is preserved in any conformal change of variables, and Carath\'eodory's theorem implies that the Riemann map $g \colon \X \to \D$ for any Jordan domain $\X$ extends homeomorphically to the boundary $\partial \X$, preserving any boundary values. Hence, we hereafter consider only the case $\X = \D$.

The implication \eqref{enum:main_thm_monotone} $\implies$ \eqref{enum:main_thm_jacobian} in Theorem \ref{thm:main} is the easy part. Indeed, since a monotone mapping $h \colon \overline{\D} \onto \overline{\Y}$ in $W^{1,2}(\D, \C)$ can be approximated with homeomorphisms in $W^{1,2}(\D, \C)$ due to \cite[Theorem 1.3]{IOmono}, the Jacobian $J_h$ cannot change sign. Here, we also used the fact that the Jacobian of any planar Sobolev homeomorphism has a constant sign~\cite[Theorem 5.22]{HKb}. The non-negativity of $J_h$ then follows from the fact that the homeomorphism $f \colon \overline{\D} \onto \overline{\Y}$ in the boundary condition is orientation-preserving.

The primary content of Theorem \ref{thm:main} is hence the implication \eqref{enum:main_thm_jacobian} $\implies$ \eqref{enum:main_thm_monotone}. It is known that solutions to the Hopf-Laplace equation~\eqref{eq:hopf-harmonic} with nonnegative Jacobian are locally Lipschitz continuous in $\X$; see~\cite{Iwaniec-Kovalev-Onninen_Duke}. Note though that these solutions need not be differentiable everywhere~\cite{CIKO}. Therefore, the main challenge is in proving that $h$ is continuous up to the boundary, and that $h^{-1}\{y\}$ is connected for every $y \in \overline{\Y}$. We in fact prove the boundary regularity of $h$  last, by leveraging the property that the sets $h^{-1}\{y\}$ are connected; the Hopf-Laplace equation is notably not required for this step. 

Hence, the core of the proof lies in showing that $h^{-1}\{y\}$ is connected for every $y \in \overline{\Y}$, where the pre-image under the homeomorphic boundary trace $h\vert\partial \D = f\vert\partial \D$ is included for $y \in \partial \Y$. First, we use degree theory and topological arguments to reduce the question into showing that for all $y \in \C$ and $r > 0$, there exists no nonempty component of $h^{-1}(B^2(y, r))$ where $J_h \equiv 0$. This topological part requires care, as we have to deal with boundary points where we have only a Sobolev boundary value. Note that the aforementioned condition is in particular not satisfied by the maps in Examples \ref{ex:1} and \ref{ex:2}. 

The remaining part is then to show that solutions of the Hopf-Laplace equation with homeomorphic Sobolev boundary values in fact satisfy the above condition. For this, we use the properties of the Hopf differential $h_z \, \overline{h_{\bar z}} \, dz \otimes dz$, combined with results on the trajectory structure of $L^1$-integrable holomorphic quadratic differentials.

\section{The Hopf-Laplace equation}

In this section, we recall some basic properties of solutions to the Hopf-Laplace equation, and their close relation with holomorphic quadratic differentials. Our main reference for the theory of holomorphic quadratic differentials is the book of Strebel \cite{Strebel_quadratic_differentials}.

\subsection{Holomorphic quadratic differentials}

Recall that a \emph{holomorphic quadratic differential} on the unit disk $\D$ is a field of symmetric complex 2-tensors on $\D$ of the form $G = g dz \otimes dz$, where the function $g \colon \D \to \C$ is holomorphic. 

A holomorphic quadratic differential $G = g dz \otimes dz$ has a \emph{critical point} at $z \in \D$ if its coefficient function $g$ has a zero at $z$. Any other points $z \in \D$ are called \emph{regular points} of $g(z) dz \otimes dz$. Note that we assume that $g$ does not have any poles; under a less restrictive definition allowing for meromorphic $g$, poles of $g$ would also generally be considered critical points of $G$. 

A smooth curve $\gamma \colon (a,b) \to \D$ is called a \emph{vertical arc} of a holomorphic quadratic differential $G$ if $G(\dot{\gamma}, \dot{\gamma}) = g \dot{\gamma}^2 < 0$ everywhere on $(a, b)$. Conversely, the curve $\gamma$ is called a \emph{horizontal arc} of $G$ if $G(\dot{\gamma}, \dot{\gamma}) > 0$ everywhere on $(a, b)$. Maximal vertical arcs are called \emph{vertical trajectories}, and maximal horizontal arcs are similarly called \emph{horizontal trajectories}.

Most notably, if $z$ is a regular point of a holomorphic quadratic differential $G$, then there exist a horizontal trajectory and a vertical trajectory of $G$ passing through $z$, and these trajectories are unique up to reparametrization; see \cite[Theorem 5.5]{Strebel_quadratic_differentials}. Two different horizontal trajectories will therefore never meet at a regular point, and the same holds for two different vertical trajectories. On the other hand, if $z$ is an isolated critical point of $G$, then $g$ has a zero at $z$ of order $k > 0$. In this case, there are $k+2$ unique horizontal trajectories and $k+2$ unique vertical trajectories which exit $z$; see e.g.\ the discussion in \cite[Section 7.1]{Strebel_quadratic_differentials}.

A vertical or horizontal trajectory which exits a critical point of $G$ at one of its endpoints is called \emph{critical}. If $G$ is not identically zero, then its critical points are isolated, and it therefore has at most countably many of them. Combined with the above description of the trajectory structure at critical points, if $G$ is not identically zero, then the union of all critical trajectories of $G$ has measure zero.

We then recall the following key properties of non-critical trajectories from \cite{Strebel_quadratic_differentials} when the quadratic differential is $L^1$-integrable.

\begin{lemma}\label{lem:noncrit_traj_boundary_lemma}
	Let $G = g dz \otimes dz$ be a holomorphic quadratic differential on the unit disk $\D$, and suppose that $g \in L^1(\D, \C)$. Let $\gamma \colon (a, b) \to \D$ be a non-critical vertical/horizontal trajectory, where $a, b \in \R \cup \{\pm \infty\}$. Then there exist $y_1, y_2 \in \partial \D$ such that $\lim_{t \to a} \gamma(t) = y_1$, $\lim_{t \to b} \gamma(t) = y_2$, and $y_1 \neq y_2$.
\end{lemma}
\begin{proof}
	The ends of the trajectory converge to well defined boundary points due to \cite[Theorem 19.6]{Strebel_quadratic_differentials}. These boundary points cannot be the same due to \cite[Theorem 19.4 a)]{Strebel_quadratic_differentials}; note that although this result is only given for horizontal trajectories (which \cite{Strebel_quadratic_differentials} refers to as just 'trajectories'), it also applies to vertical trajectories, since replacing $G$ with $-G$ swaps its horizontal and vertical trajectories with each other.
\end{proof}

\subsection{The Hopf-Laplace equation}\label{subsect:hopf_harm_prelims}

Suppose that $h \in W^{1,2}_\loc(\D, \C)$ satisfies the Hopf-Laplace equation \eqref{eq:hopf-harmonic} in a weak sense. We denote $\varphi = \partial_z h  \overline{\partial_{\overline{z}} h}$, in which case the equation reads as  $\partial_{\overline{z}} \varphi = 0$; this particular $\varphi$ is also called the \emph{Hopf product} of $h$. 

Since $\varphi \in L^1_\loc(\D, \C)$, the Hopf-Laplace equation therefore implies that the Hopf product $\varphi$ is weakly holomorphic. Since $\Delta = 4 \partial_z \partial_{\overline{z}}$, the map $\varphi$ is also weakly harmonic, and hence smooth by Weyl's lemma. Hence, the Hopf-Laplace equation is equivalent with requiring that $\varphi$ is a holomorphic map.

The Hopf product $\varphi$ of the solution $h$ therefore defines a holomorphic quadratic differential $H = \varphi dz \otimes dz$ on $\D$. A notable property of $H$ is that its vertical and horizontal arcs travel exactly in the directions of minimal and maximal stretch of $Dh$. Indeed, we have $Dh(z)e^{i\theta} = (\partial_z h) e^{i\theta} + (\partial_{\overline{z}} h) e^{-i\theta}$, and therefore 
\begin{align*}
	\smallabs{Dh(z)e^{i\theta}}^2 
	&= \abs{\partial_z h}^2 + \abs{\partial_{\overline{z}} h}^2 
	+ 2\Re [(\partial_z h) (\overline{\partial_{\overline{z}} h}) e^{2i\theta}]\\
	&= \abs{\partial_z h}^2 + \abs{\partial_{\overline{z}} h}^2 
	+ 2\Re [H(z)(e^{i\theta}, e^{i\theta})].
\end{align*}
Since $\smallabs{H(z)(e^{i\theta}, e^{i\theta})}$ is independent of $\theta$, the above quantity is therefore respectively minimized/maximized when the value of $H(z)(e^{i\theta}, e^{i\theta})$ lies on the negative/positive real axis.

\subsection{Continuity inside $\D$}
We then recall an important interior regularity result, which acts as essentially the starting point for the proof of our main result. Namely, suppose that $h \in W^{1,2}(\D, \C)$ is a solution to the Hopf-Laplace equation with non-negative Jacobian. Then it follows that $h$ is locally Lipschitz continuous inside $\D$. This was shown by Iwaniec, Kovalev and Onninen in \cite{Iwaniec-Kovalev-Onninen_Duke}.

\begin{thm}[{\cite[Theorem 1.3]{Iwaniec-Kovalev-Onninen_Duke}}]\label{thm:hopf_harm_lipschitz}
	Let $\Omega \subset \C$ be open, and let $h \in W^{1,2}(\Omega, \C)$ with $J_h \geq 0$ almost everywhere. Suppose that the Hopf product $\varphi = \partial_z h \overline{\partial_{\overline{z}} h}$ is bounded and H\"older continuous. Then $h$ is locally Lipschitz continuous.
\end{thm}

Most notably, we obtain that a solution of the Hopf-Laplace equation $h$ with $J_h \geq 0$ almost everywhere is continuous inside $\D$, although we do not  yet have continuity up to the boundary. We also obtain that $h \colon \D \to \C$ satisfies the Lusin (N) -condition; that is, that for every set of measure zero $E \subset \D$, also the image set $h(E)$ has measure zero.

\subsection{Domains with zero Jacobian}

Suppose that $h$ is a solution of the Hopf-Laplace equation in a domain $\Omega \subset \C$, and that $J_h$ is zero almost everywhere in $\Omega$. Let $H$ be the corresponding holomorphic quadratic differential. Since the vertical trajectories of $H$ travel in the direction of minimal stretch of $Dh$, and since $J_h \equiv 0$, it is to be expected that the derivative of $h$ along a vertical trajectory is zero. Therefore, $h$ should be constant along vertical trajectories if its Jacobian is zero.

Indeed, a precise version of the above argument has been given by Iwaniec and Onninen in \cite[Lemma 2.6]{Iwaniec-Onninen_Hopf-Harmonics}.  The exact result they show is as follows; note that their original statement includes an assumption that $h$ is locally Lipschitz, but this is unnecessary due to Theorem \ref{thm:hopf_harm_lipschitz}.

\begin{lemma}\label{lem:Hopf_harm_vert_traj_constant}
	Let $\Omega \subset \C$ be open, and suppose that $f \in W^{1,2}_\loc(\Omega, \C)$ satisfies the Hopf-Laplace equation $\partial_{\overline{z}} \varphi = 0$, where $\varphi = (\partial_z f) ( \overline{\partial_{\overline{z}} f})$. If $J_f(x) = 0$ a.e.\ in $\Omega$, then $f$ is constant on every vertical arc of the Hopf differential $H = \varphi dz \otimes dz$.
\end{lemma}

\section{Sobolev boundary conditions}

In this section, we recall some preliminaries related to trace maps and the extension of boundary homeomorphisms on the unit disk. We finish with a lemma which essentially extracts the required topological information provided by our homeomorphic Sobolev boundary condition.

\subsection{Trace maps and Sobolev homeomorphisms}\label{sec:traceSobolevmaps}

Suppose that $f \in W^{1,2}(\D)$. There exists a bounded \emph{trace operator} $T \colon W^{1,2}(\D) \to L^n(\S^1)$, $\S^1=\partial \D$, such that for every continuous Sobolev map $g \in W^{1,2}(\D) \cap C(\overline{\D})$, we have $Tg = g\vert\S^{1}$ a.e.\ on $\S^{1}$. We can hence define a \emph{trace map} $f\vert\S^{1} := Tf \in L^2(\S^{1})$ of $f$. The definition immediately extends to $f \in W^{1,2}(\D, \C)$ by $T(f_1 + if_2) = T(f_1) + i T(f_2)$.

We then recall that the \emph{Sobolev space with vanishing boundary values} $W^{1,2}_0(\D)$ is defined as the closure of  the space of smooth compactly supported functions $C^\infty_0(\D)$ in $W^{1,2}(\D)$. A similar definition also applies for $W^{1,2}_0(\D, \C)$. The trace operator provides an alternate characterization of $W^{1,2}_0(\D)$: namely, $f \in W^{1,2}_0(\D)$ if and only if $f \in W^{1,2}(\D)$ and $f\vert \S^{1} \equiv 0$. Consequently, two functions $f, g \in W^{1,2}(\D)$ have the same trace if and only if $f - g \in W^{1,2}_0(\D)$.

The image of the trace operator in $L^2(\S^{1}, \C)$ is precisely the \emph{fractional Sobolev space} $W^{1/2,2}(\S^{1}, \C)$. This space consists of exactly the maps $f \in L^2(\S^{1}, \C)$ which satisfy the \emph{Douglas condition}~\cite{Do}
\[
	\int_{\S^{1} \times \S^{1}} \abs{\frac{f(z) - f(z')}{z-z'}}^2 \dd z \dd z' < \infty.
\]

The Douglas condition has a key role in the theory of Sobolev homeomorphic extensions. Indeed, suppose that $\Y \subset \C$ is a Lipschitz Jordan domain, and $f \colon \S^1 \to \partial \Y$ is a homeomorphism which satisfies the Douglas condition. Then $f$ can be extended to a homeomorphism $f \colon \overline{\D} \to \overline{\Y}$ with $f \in W^{1,2}(\D, \C)$ , see \ \cite[pp. 2--3]{Koski-Onninen_Sobolev-homeo}. We hence obtain the following corollary.

\begin{cor}\label{cor:homeomorphic_extension}
	Let $g \in W^{1,2}(\D, \C)$ be such that the trace $g\vert\S^1$ is a homeomorphism onto the boundary of a Lipschitz domain $\Y$. Then there exists a homeomorphism $f \colon \overline{\D} \to \overline{\Y}$ such that $g-f \in W^{1,2}_0(\D, \C)$.
\end{cor}

\begin{rem}\label{rem:Sobolev_extension}
	When $\Y$ is a Lipschitz Jordan domain, any homeomorphism $f \colon \overline{\D} \to \overline{\Y}$ with $f \in W^{1,2}(\D, \C)$ can be extended to a homeomorphism $\tilde{f} \colon \C \to \C$ with  $\tilde{f} \in W^{1,2}_\loc(\C, \C)$ by a standard reflection argument. 
\end{rem}

\subsection{Preliminaries on conformal capacity}

Suppose that $U \subset \C$ is an open domain, and $K \subset U$ is compact. The \emph{conformal capacity} $\capac(K, U)$ of the condenser $(K, U)$ is the infimum
\begin{equation}\label{eq:capacity}
	\inf_u \int_U \abs{\nabla u(z)}^2 \dd z,
\end{equation}
where the infimum is taken over all $u \in C^\infty_0(U)$ with $u\vert K \geq 1$. The conformal capacity is notably preserved in conformal transformations.

\begin{rem}\label{rem:admissible_functions}
We note that the same value of $\capac(K, U)$ is obtained if we instead take the infimum over $u \in W^{1,2}_0(U) \cap C(U)$ with $u\vert K \geq 1$. We will hence call any such $u$ \emph{admissible} for $(K, U)$. The proof of this fact is a standard approximation argument, and is explained e.g.\ in \cite[pp.\ 27--28]{Heinonen-Kilpelainen-Martio_book}.
\end{rem}

We then recall several standard results on conformal capacity. The first is a symmetrization theorem for capacities; see e.g.\ \cite[Section 7.16]{Vuorinen_book}. Namely, suppose that $K \subset \C$ is compact, and that $L$ is a ray originating from the origin. The \emph{symmetrization} $\sym_L(K)$ of $K$ is defined as follows: if $K \cap \S^{1}(r) = \emptyset$, then $\sym_L(K) \cap \S^{1}(r) = \emptyset$, and if $K \cap \S^{1}(r) \neq \emptyset$, then $\sym_L(K) \cap \S^{1}(r)$ is the closed circular arc around $L$ with the same $1$-dimensional Hausdorff measure as $K \cap \S^{1}(r)$. The symmetrization theorem gives a lower bound for $\capac(K, U)$ using a symmetrized condenser.

\begin{thm}\label{thm:capacity_symmetrization}
	Let $U \subset \C$ be open and let $K \subset U$ be compact. Let $L$ be a ray originating from the origin. Then
	\[
		\capac(K, U) \geq \capac(\sym_L(K), \C \setminus \sym_{-L}(\C \setminus U)).
	\]
\end{thm}

The second fact we require is that the capacity of line segments approaching the boundary of $\D$ tends to infinity. The following formulation of this result follows from the basic properties of the capacity of the Gr\"otszch ring (see e.g.\ \cite[Section 7.18, (7.23)]{Vuorinen_book}) by using a conformal change of variables.

\begin{lemma}\label{lem:line_segment_capacity}
	Let $a \in (0,1)$, and let $\eps \in (0, 1-a)$. Then
	\[
		\lim_{\eps \to 0} \capac\bigl([a, 1-\eps], \D\bigr) = \infty.
	\]
\end{lemma}

Finally, we recall that the conformal capacity is monotone: if $K \subset K' \subset U' \subset U$, then $\capac(K, U) \leq \capac(K', U')$. This is merely since any admissible function for $(K', U')$ is also admissible for $(K, U)$. By combining this monotonicity property with Theorem \ref{thm:capacity_symmetrization} and Lemma \ref{lem:line_segment_capacity}, the following important corollary immediately follows.

\begin{cor}\label{cor:capacity_goes_to_infinity}
	Let $K \subset \D$ be closed in $\D$. Suppose that there exists $r_0 \in [0, 1)$ such that for every $r \in (r_0, 1)$, we have $K \cap \S^{1}(r) \neq \emptyset$. Then
	\[
		\lim_{r \to \infty} \capac(K \cap \overline{\D(r)}, \D) = \infty.
	\]
\end{cor}

\subsection{Topological lemmas on Sobolev boundary conditions}

Our key application of conformal capacity is a topological consequence of Sobolev boundary conditions. We use the following lemma numerous times throughout the text as essentially our replacement to boundary continuity in topological arguments.

\begin{lemma}\label{lem:capacity_to_topology_lemma}
	Suppose that $h \in C(\D, \C) \cap W^{1,2}(\D, \C)$, and that $h - f \in W^{1,2}_0(\D, \C)$ for some continuous $f \in C(\overline{\D}, \C)$. Then for every connected set $C \subset \D$ with $\overline{C} \cap \S^1 \neq \emptyset$, we have
	\[
		f(\overline{C} \cap \S^1) \cap \overline{h(C)} \neq \emptyset.
	\]
\end{lemma}
\begin{proof}
	Suppose towards contradiction that the intersection is empty. It follows that $\overline{C} \cap \S^1 \cap f^{-1}\overline{h(C)} = \emptyset$; that is, the set $\overline{C} \cap f^{-1} \overline{h(C)}$ does not meet $\S^1$. Since $f^{-1} \overline{h(C)}$ and $\overline{C}$ are compact, it follows that $\overline{C} \cap f^{-1} \overline{h(C)}$ is a compact subset of $\D$. Consequently, there exists $r_0 \in (0,1)$ such that $\overline{C} \cap f^{-1} \overline{h(C)} \subset \D(r_0)$.
	
	In particular, we now have that $\overline{C} \setminus \D(r_0)$ does not meet $f^{-1} \overline{h(C)}$. It follows that $f(\overline{C} \setminus \D(r_0)) \cap \overline{h(C)} = \emptyset$. Since $f(\overline{C} \setminus \D(r_0))$ is compact and $\overline{h(C)}$ is closed, these two sets must therefore have positive distance $d_0 := d(f(\overline{C} \setminus \D(r_0)), \overline{h(C)}) > 0$ from each other. Note that since $h$ is continuous in $\D$, we have $h((\overline{C} \setminus \D(r_0)) \cap \D) \subset h(\overline{C} \cap \D) \subset \overline{h(C)}$. Hence, we have that $\abs{f - h} \geq d_0$ on $(\overline{C} \setminus \D(r_0)) \cap \D$.
	
	We then define
	\[
		g = \frac{2\sqrt{2}}{d_0} \max(\abs{\Re(h-f)}, \abs{\Im(h-f)}).
	\]
	The map $g$ is continuous, $g \geq 1$ on $(\overline{C} \setminus \D(r_0)) \cap \D$, and $g \in W^{1,2}_0(\D)$ by e.g.\ \cite[Lemma 1.23]{Heinonen-Kilpelainen-Martio_book}. Hence,
	\[
		\capac(\overline{C} \cap (\overline{\D(r)} \setminus \D(r_0)), \D) \leq \norm{\nabla g}_{L^n}^n < \infty
	\]
	when $r \in (r_0, 1)$; see Remark \ref{rem:admissible_functions}.
	
	However, since $C$ is a continuum in $\D$ with $\overline{C} \cap \S^1 \neq \emptyset$, if we have $C \cap \S^1(r) \neq \emptyset$ for some $r \in (0, 1)$, we must then have $C \cap \S^1(r') \neq \emptyset$ for all $r' \in (r, 1)$. Indeed, otherwise the two components of $\D \setminus \S^1(r')$ would yield a separation of $C$. It follows that $(\overline{C} \setminus \D(r_0))$ meets $\S^1(r)$ for all sufficiently large $r \in (0, 1)$. Hence, by Corollary \ref{cor:capacity_goes_to_infinity}, we have
	\[
		\lim_{r \to 1} \capac(\overline{C} \cap (\overline{\D(r)} \setminus \D(r_0)), \D) = \infty.
	\]
	This is a contradiction, and the claim is hence proven.
\end{proof}

We point out a corollary of Lemma \ref{lem:capacity_to_topology_lemma}, which is of significant importance when we consider degree theory.

\begin{cor}\label{cor:capacity_to_topology_preimage}
	Suppose that $h \in C(\D, \C) \cap W^{1,2}(\D, \C)$, and that $h - f \in W^{1,2}_0(\D, \C)$ for some continuous $f \in C(\overline{\D}, \C)$. Let $A \subset \C$ be such that $\overline{A} \cap f(\S^1) = \emptyset$, and let $C$ be a connected component of $h^{-1} A$. Then $\overline{C} \subset \D$.
\end{cor}
\begin{proof}
	Suppose towards contradiction that $\overline{C} \cap \S^1 \neq \emptyset$. Then Lemma \ref{lem:capacity_to_topology_lemma} applies, and yields $f(\overline{C} \cap \S^1) \cap \overline{h(C)} \neq \emptyset$. However, since $C$ is contained in the pre-image of $A$, we have $f(\overline{C} \cap \S^1) \cap \overline{h(C)} \subset f(\S^1) \cap \overline{A} = \emptyset$; a contradiction, which proves the claim.
\end{proof}

\section{$K$-oscillation property}
In this section we recall that weakly monotone mappings with finite Dirichlet integral are continuous. The notion of weak monotonicity is due to Manfredi~\cite{Ma}, and is a  powerful tool when dealing with continuity properties of Sobolev functions. Roughly speaking, if $U \subset \C$ is open and $u \colon  U \to \C$ is a Sobolev mapping, then $u$ is weakly monotone if both coordinate functions $\Re u$ and $\Im u$ satisfy the maximum and minimum principles in the Sobolev sense in all balls $B \subset U$. For a more precise definition, see~\cite[p. 532]{AIMb} or~\cite[p. 395]{Ma}. 

We consider a slightly more general notion compared to the usual weak monotonicity given in~\cite{Ma}.
\begin{defn} \label{diamdefin}
Let $U \subset \C$ be open and $K\ge 1$.  A Sobolev mapping $H\colon U\to\C$ is said to have the $K$-\emph{oscillation property} if for every $z \in U$ and almost every $r \in (0, d(z, \C \setminus U))$ we have
\[
	\underset{B^2(z,r)}{\osc} H  \leq K\underset{\partial B^2(z,r)}{\osc} H \, . 
\]
\end{defn}
Here $\underset{K}{\osc} \, H := \sup \{\abs{H(x)-H(x')} \colon x,x'\in  K \}$.

\begin{rem}
	We note that continuous monotone Sobolev maps satisfy the $1$-oscillation property. Indeed, the $1$-oscillation property is satisfied by homeomorphisms and is preserved in uniform limits, which implies the property for monotone maps by e.g.\ \cite[Remark 3.1]{Iwaniec-Onninen_Hopf-Harmonics}.
	
	The converse does not hold, as unlike in the case of a continuous monotone map, weakly monotone maps and maps with the $K$-oscillation property may in fact cause some types of folding. For a standard example, let $Q =(0,2) \times (0,2)$ and
	\[
		g \colon \overline{Q} \onto \overline{Q} \qquad g(x,y) = 
		\begin{cases} (2x,y) \quad & \textnormal{if } 0\le x \le  1 \\ 
			(-2x+4, y) \quad & \textnormal{if } 1< x \le  2 \, . 
		\end{cases} 
	\]
	Then $g$ is continuous. Moreover, $g$ has the $1$-oscillation property in $Q$, and $g$ is also weakly monotone in the sense of \cite{Ma}. However, $g$ is not monotone.
\end{rem}
The proof of the following continuity result is standard; see e.g.\ \cite[Theorem 20.1.6]{AIMb} for the weakly monotone version.
\begin{lemma}\label{lem:oscwm} Let  $H\in W^{1,2} (B_R, \C)$ satisfy the $K$-oscillation property with $K\ge 1$ in the ball $B_R=B^2(z,R)$. Then for $r\in (0,R/2)$ we have
\[\left[\underset{B^2(z,r)}{\osc} H \right]^2 \le \frac{C}{\log \left( \frac{R}{r}\right)} \int_{B_R} \abs{DH}^2\]
for some constant $C \ge 1$ depending only on $K$. In particular, $H$ is continuous.
\end{lemma}
\begin{proof}
Since $H\in W^{1,2} (B_R, \C)$, the coordinate functions $\Re H$ and $\Im H$ are absolutely continuous on almost all of the
circles $\partial B_t=S_t$, $0<t<R$. It follows that, for almost every $t\in (0,R)$,
\begin{equation}\label{eq:osc}
 \underset {S_t}{\osc} H  \le C \int_{S_t} \abs{DH} \le C \, t^\frac{1}{2} \left( \int_{S_t} \abs{DH}^2 \right)^\frac{1}{2} \, .
 \end{equation}
For almost every $t \in (r,R)$ we have
\[ \underset{B_r}{\osc} H \le  \underset{B_t}{\osc} H \le K \,   \underset{S_t}{\osc} H \, . \]
Combining this with~\eqref{eq:osc} we obtain
\[ \left(  \underset{B_r}{\osc} H \right)^2 \int_r^R \frac{\dtext t}{t} \le C \int_{B_R \setminus B_r} \abs{DH}^2 \, , \]
as desired.
\end{proof}

\section{Degree theory}

Suppose that $U \subset \R^n$ is open, and that $h \in C(\overline{U}, \R^n)$. Then the classical \emph{Brouwer degree} $\deg(h, y, U)$ is well defined for every $y \notin h(\partial U)$. For further details on classical degree theory, we refer to e.g.\ \cite{Fonseca-Gangbo-book}.

In our case, however, we need to be a bit more careful with the use of degree theory. This is because we are mostly dealing with mappings which are not a priori assumed to be continuous up to the boundary. In our main result, Theorem \ref{thm:main}, it is only assumed that $h$ is continuous in $\D$ and has a continuous Sobolev trace on the boundary. We note that there does exist literature on degree theory in settings more general than ours, such as the degree theory of Nirenberg and Brezis \cite{Nirenberg-Brezis_BMO-Degree-boundary} for $\vmo$-maps. However, we have found no account which includes all the results we require, and we found it easier to derive the desired results in our specific setting using the classical degree theory.

\subsection{Sobolev boundary conditions and topology}

We now consider a $h \in C(\D, \C) \cap W^{1,2}(\D, \C)$ such that $h$ satisfies the Lusin (N) -condition, and $h - f \in W^{1,2}_0(\D, \C)$ for some continuous $f \in C(\overline{\D}, \C)$. It hence follows that the trace $h\vert\S^1$ equals the restriction $f \vert \S^1$, and is hence continuous. However, $h$ is not necessarily continuous up to the boundary.

By using Corollary \ref{cor:capacity_to_topology_preimage}, we are able to define a topological degree for $h$ with Sobolev boundary values in the following situation.

\begin{lemma}\label{lem:degree_in_preimage}
	Suppose that $h \in C(\D, \C) \cap W^{1,2}(\D, \C)$ satisfies the Lusin (N) -condition, and that $h - f \in W^{1,2}_0(\D, \C)$ for some continuous $f \in C(\overline{D}, \C)$. Let $V \subset \C$ be a bounded, connected open set such that $d_0 := d(f(\S^1), \overline{V}) > 0$ and $m_2(\partial V) = 0$. Let $U$ be a connected component of $h^{-1} V$.
	
	Then $\overline{U} \subset \D$, and for every $y \in V$, $\deg(h, y, U)$ is well defined. Moreover, for every such $y$ we have
	\[
		\deg(h, y, U) = \frac{1}{m_2(V)} \int_U J_h(z) \dd z.
	\]
	In particular, the function $y \mapsto \deg(h, y, U)$ is constant on $V$.
\end{lemma}
\begin{proof}
	We note that since $\D$ is locally connected and $h^{-1} V$ is open, the components of $h^{-1} V$ are also open. Consequently, $U$ is open and connected. Since $f(\S^1) \cap \overline{V} = \emptyset$, it follows from Corollary \ref{cor:capacity_to_topology_preimage} that $\overline{U} \subset \D$.
	
	We note that we must necessarily have $h(\partial U) \subset \partial V$. Indeed, since $U$ is a connected component of $h^{-1} V$, $U$ is also closed in $h^{-1} V$, and since $U$ is also open in $\C$, no element of $\partial{U}$ can be in $h^{-1} V$. Therefore, $h(\partial U) \subset \C \setminus V$. On the other hand, since $\partial U \subset \D$ and $h$ is continuous in $\D$, we have $h(\partial U) \subset \overline{h(U)} \subset \overline{V}$.
	
	Now, since $h$ is continuous on $\overline{U}$ and $V \cap h(\partial U) = \emptyset$, the degree $\deg(h, y, U)$ is well defined for any $y \in V$. Since $V$ is connected and doesn't meet $h(\partial U)$, it also follows that $y \mapsto \deg(h, y, U)$ is constant in $V$; see e.g.\ \cite[Theorem 2.3 (3)]{Fonseca-Gangbo-book}. 
	
	For the Jacobian formula, note that if $y \notin \overline{V}$, we must therefore have $\deg(h, y, U) = 0$ by e.g.\ \cite[Theorem 2.1]{Fonseca-Gangbo-book} Since $\overline{U} \subset \D$, $h(\partial U) \subset \partial V$, and since $h$ satisfies the Lusin (N) -property, we therefore obtain using e.g.\ \cite[Proposition 5.25 and Remark 5.26 (ii)]{Fonseca-Gangbo-book} that
	\[
	\int_U J_f(z) \dd z = \int_\C \deg(h, z', U) \dd z' = m_2(V) \deg(h, y, U)
	\]
	for every $y \in V$, which implies the given Jacobian formula.
\end{proof}

This lemma allows us to define the degree of $h$ in a pre-image.

\begin{defn}\label{def:degree_in_preimage}
	Let $h$, $f$, and $V$ be as in Lemma \ref{lem:degree_in_preimage}. Let $U_1, U_2, \dots$ be the connected components of $h^{-1} V$. Note that since $U_i$ are open and disjoint, there are at most countably many of them. We define
	\[
		\deg(h, y, h^{-1} V) = \sum_{i} \deg(h, y, U_i)
	\]
	for any $y \in V$. This sum converges to a finite value, since the Jacobian formula of Lemma \ref{lem:degree_in_preimage} implies that
	\[
		\sum_{i} \abs{\deg(h, y, U_i)}
		\leq \frac{1}{m_2(V)} \sum_i \int_{U_i} \abs{J_h(z)} \dd z
		\leq \frac{1}{m_2(V)} \int_{\D} \abs{Dh(z)}^2 \dd z
		< \infty.
	\]
\end{defn}

The remaining tool we require is that $\deg(h, y, h^{-1} V) = \deg(f, y, f^{-1} V)$ when $h$ and $f$ have the same trace. We restrict to the case where $V$ is a ball and $f$ is a continuous Sobolev map in the entire plane, as that is enough for us.

\begin{lemma}\label{lem:degree_equality lemma}
	Suppose that $h \in C(\D, \C) \cap W^{1,2}(\D, \C)$ satisfies the Lusin (N) -condition, and that $h - f \in W^{1,2}_0(\D, \C)$ for some $f \in C(\C, \C) \cap W^{1,2}_\loc(\C, \C)$ which also satisfies the Lusin (N) condition. Let $y \in \C$ and $r > 0$ be such that $d(f(\S^1), \overline{B^2(y, r)}) > 0$. Then
	\[
		\deg(h, y, h^{-1} B^2(y, r)) = \deg(f, y, \D \cap f^{-1} B^2(y, r)) \, . 
	\]
\end{lemma}
\begin{proof}
	Note that we may assume $f \in W^{1,2}(\C, \C)$ by multiplying it with some $\eta \in C^\infty_0(\C)$ with $\eta \vert B^2(0, 2) \equiv 1$. By post-composing with an affine map, we may also assume that $y = 0$ and $r = 1$, and therefore $B^2(y, r) = \D$. We may extend $h$ to $\C$ by setting $h\vert(\C\setminus\D)= f\vert(\C\setminus\D)$; this extension is in $W^{1,2}(\C, \C)$, although it isn't necessarily continuous on $\S^1$.
	
	We then let $\psi_j \in C^\infty(\C, \C)$ be a smooth approximation of the radial retraction to $\D$; that is $\psi_j\vert(\D(1-j^{-1})) = \id$, $(\psi_j \vert (\C \setminus \D))(z) = z/\abs{z}$, and $\abs{D\psi_j} \leq 2$. We define $h_j = \psi_j \circ h$ and $f_j = \psi_j \circ f$. Since $h$ is a smooth 2-Lipschitz map, we have $h_j, f_j \in W^{1,2}(\C, \C)$.
	
	Since $W^{1,2}(\C, \C) = W^{1,2}_0(\C, \C)$, we have
	\[
		\int_\C J_{h_j}(z) \dd z = \int_\C J_{f_j}(z) \dd z = 0.
	\]
	Since $h_j \equiv f_j$ in $\C \setminus \D$, their Jacobians also coincide almost everywhere in $\C \setminus \D$. It follows that
	\[
		\int_\D J_{h_j}(z) \dd z = \int_\D J_{f_j}(z) \dd z.
	\]
	However, we have $J_{h_j} \to J_h \chi_{h^{-1}\D}$, since $J_{\psi_j} = 1$ in $\D(1-j^{-1})$ and $J_{\psi_j} = 0$ in $\C \setminus \D$. Similarly, $J_{f_j} \to J_f \chi_{f^{-1}\D}$, and we also have $\abs{J_{h_j}} \leq 4\abs{Dh}^2$ and $\abs{J_{f_j}} \leq 4\abs{Df}^2$. Hence, by dominated convergence, we get
	\[
		\int_{\D \cap h^{-1} \D} J_{h}(z) \dd z = \int_{\D \cap f^{-1} \D} J_{f}(z) \dd z.
	\]
	Now we obtain the claim by the fact that the degrees of $f$ and $h$ are constant for $y' \in B^2(y, r) = \D$, combined with the Jacobian formula for the degree. In the case of $h$, the versions of these results we use are precisely the ones proven in Lemma \ref{lem:degree_in_preimage}.
\end{proof}

\section{Connectedness of fibers}

We now begin the process of proving Theorem~\ref{thm:main}. As discussed in the introduction, we may assume $\X = \D$ by a conformal change of variables. Moreover, the ``only if''-part of the theorem was also explained to follow from the fact that a monotone $h \in W^{1,2}(\D, \C)$-map is a $W^{1,2}(\D, \C)$-limit of homeomorphisms, and therefore $J_h$ cannot change sign. Hence, only the ``if''-part remains to be proven here:

\begin{thm}\label{thm:hopf_monotonicity_upgraded}
 Let $\Y$ be Lipschitz regular and $f \colon \overline{\D} \onto \overline{\Y} $ an orientation-preserving homeomorphism in  $W^{1,2} (\D, \C)$. Suppose that a mapping $h \colon \D \to \C$ satisfies the Hopf-Laplace equation, $J_h \ge 0$ almost everywhere in $\D$, and $h\in f+W^{1,2}_0(\D, \C)$. Then $h$ extends to a continuous monotone mapping from $\overline{\mathbb D}$ onto $\overline{\Y}$ (up to redefining $h$ in a set of measure zero).
\end{thm}

Our goal in this section is to show that if $\Y$ and $h$ satisfy the assumptions of Theorem \ref{thm:hopf_monotonicity_upgraded}, then $\Y \subset h(\D) \subset \overline{\Y}$, and the pre-image $h^{-1}\{y\} \cup (h\vert\S^1)^{-1}\{y\}$ is connected for every $y \in \overline{\Y}$, where $h\vert\S^1 = f\vert\S^1$ is the continuous representative of the trace map of $h$. Note that including the $(h\vert\S^1)^{-1}\{y\}$-part is important when $y \in \partial\Y$; the set $h^{-1}\{y\}$ can e.g.\ consist of two disjoint paths converging to the same point of $\S^1$, in which case the interior part $h^{-1}\{y\}$ of the pre-image is not connected but $h^{-1}\{y\} \cup (h\vert\S^1)^{-1}\{y\}$ instead is. We also note that we do not yet show continuity of $h$ up to the boundary; hence, we refrain from referring to this property of $h$ as monotonicity.

\subsection{A topological preliminary result}

We begin with a topological lemma, which essentially transfers some properties from the components of $f^{-1}\{y\}$ to the components of $f^{-1}B^2(y, r)$ with small enough $r$.

\begin{lemma}\label{lem:preimage_lemma_generalized}
	Let $f \colon \D \to \C$ be a continuous map. Let $y \in \C$, and let $K \subset \D$ be a component of $f^{-1}\{y\}$ such that $\overline{K} \subset \D$. Then 
	\begin{enumerate}
		\item\label{enum:component_boundary} there exists $r_K > 0$ such that for every $r \in (0, r_K)$, the component $U_r$ of $f^{-1} B^2(y, r)$ containing $K$ satisfies $\overline{U_r} \subset \D$;
		\item\label{enum:component_pair} if $K' \neq K$ is a another component of $f^{-1}\{y\}$ such that $\overline{K'} \subset \D$, then there exists $r_{K, K'} > 0$ such that for every $r \in (0, r_{K, K'})$, the sets $K$ and $K'$ are contained in different components of $f^{-1} B^2(y, r)$.
	\end{enumerate}
\end{lemma}
\begin{proof}
	We begin by proving \eqref{enum:component_boundary}. Note that $K$ is indeed always contained in a single component of $f^{-1} B^2(y, r)$, since the components would otherwise yield a separation of $K$ due to being open. Suppose then to the contrary that we have a sequence $r_1 > r_2 > \ldots$ such that $r_i \to 0$ and $\overline{U_{r_i}} \cap \S^1 \neq \emptyset$. 
	
	Note that $\overline{U_{r_i}}$ are connected, since closures of connected subsets are always connected. We let $C = \bigcap_i \overline{U_{r_i}}$. Then $C$ contains $K$ since every $\overline{U_{r_i}}$ contains $K$. The set $C$ also meets $\S^1$; indeed, the sets $\S^1 \cap \overline{U_{r_i}}$ form a decreasing sequence of non-empty compact sets, so their intersection is non-empty. Moreover, $C$ is connected, since it is the intersection of a decreasing sequence of compact connected sets; see e.g.\ \cite[Corollary 6.1.19]{Engelking_topology}.
	
	Since $f$ is continuous, we have $f(\overline{U_{r_i}} \cap \D) \subset \overline{B^2(y, r_i)}$ for every $i$, and therefore $f(C \cap \D) \subset \{y\}$. Since $K$ is a connected component of $f^{-1}\{y\}$ it follows that the connected component of $C \cap \D$ containing $K$ must in fact equal $K$. Hence, we find an open neighborhood $U$ of $K$ such that $\overline{U}$ is disjoint from $(C \cap \D) \setminus K$. Since $\overline{K} \subset \D$, we also find an open neighborhood $V$ of $K$ with $\overline{V} \subset \D$. The sets $U \cap V$ and $\C \setminus (\overline{U} \cup \overline{V})$ then yield a separation of $C$, which contradicts the connectedness of $C$. Hence, the proof of \eqref{enum:component_boundary} is complete.
	
	The proof of \eqref{enum:component_pair} is similar. Suppose that we have a sequence $r_1 > r_2 > \ldots$ such that $r_i \to 0$ and a component $U_i$ of $f^{-1} B^2(y, r_i)$ meets both $K$ and $K'$, in which case we in fact have $K \cup K' \subset U_i$. Using the previous part \eqref{enum:component_boundary}, we may assume that $r_i$ are small enough that $\overline{U_i} \subset \D$.
	
	We then define $C = \bigcap_i \overline{U_{i}}$. Since $\overline{U_i}$ are a descending sequence of compact connected sets, we again have that $C$ is a compact connected set. Since $K \cup K' \subset U_{i}$ for every $i$, we have $K \cup K' \subset C$. However, the continuity of $f$ and the fact that $\overline{U_i} \subset \D$ implies that $f(\overline{U_{i}}) \subset \overline{B^2}(y, r_i)$, from which it follows that $f(C) \subset \{y\}$. Therefore, $C$ is contained in a single component of $f^{-1}\{y\}$, which is only possible if $K = K'$, completing the proof of \eqref{enum:component_pair}.
\end{proof}

\subsection{Application of degree theory}

In this subsection, we prove the connectedness of the fibers $h^{-1}\{y\}$ with an extra assumption. Namely, we do not assume that $h$ is a solution of the Hopf-Laplace equation, but instead we assume that there exists no open ball $B \subset \C$ and non-empty component $U$ of $h^{-1}B$ such that $J_h \equiv 0$ a.e.\ in $U$. This extra assumption allows for relatively natural proofs using degree theory. Later, in the next subsection, we then proceed to eliminate this extra assumption using the Hopf-Laplace equation.

\begin{lemma}\label{lem:degree_lemma_interior}
	Let $h \in C(\D, \C) \cap W^{1,2}(\D, \C)$ be such that $J_h \geq 0$ a.e.\ in $\D$ and $h$ satisfies the Lusin (N) condition. Suppose that $h - f \in W^{1,2}_0(\D)$, where $f \colon \overline{\D} \to \overline{\Y}$ is an orientation preserving homeomorphism onto a Lipschitz Jordan domain $\Y$. Then $\Y \subset h(\D)$.
	
	Moreover, suppose also that there exists no non-empty component $U$ of the pre-image $h^{-1} B$ of an open ball $B$ such that $J_f \equiv 0$ a.e.\ in $U$. Then $h^{-1}\{y\} = \emptyset$ for every $y \in \C \setminus \overline{Y}$ and $h^{-1}\{y\}$ is connected for every $y \in \Y$.
\end{lemma}

\begin{proof}
	Note that we may extend $f$ to a continuous map in $W^{1,2}(\C, \C)$, by first taking the homeomorphic extension $f \in W^{1,2}_\loc(\C, \C)$ of Remark \ref{rem:Sobolev_extension}, and then multiplying with a suitable smooth cutoff function.
	
	Let $y \in \C \setminus \partial \Y$, and let $r \in (0, d(y, \partial \Y))$. We consider first the case $y \notin \overline{\Y}$. Suppose towards contradiction that $h^{-1}\{y\} \neq \emptyset$. Then $h^{-1} B^2(y, r)$ is nonempty. Let $U_1, U_2, \dots$ be the connected components of $h^{-1} B^2(y, r)$, in which case at least one nonempty component exists.
	
	By Lemma \ref{lem:degree_in_preimage}, every component $U_i$ is compactly contained in $\D$, and $\deg(h, U_i, y)$ is well defined. Since $J_h \geq 0$ a.e.\ in $\D$, we have $\deg(h, U_i, y) \geq 0$ for every $i$ by the Jacobian formula of Lemma \ref{lem:degree_in_preimage}. Moreover, since we may extend $f$ to $C(\C, \C) \cap W^{1,2}_\loc(\C, \C)$ by Remark \ref{rem:Sobolev_extension}, and since the extension satisfies the Lusin (N) condition due to e.g.\ \cite[Theorem 4.9]{HKb}, we may apply Lemma \ref{lem:degree_equality lemma} and obtain
	\[
		\sum_i \deg(h, y, U_i) = \deg(f, y, \D \cap f^{-1}B^2(y, r)) = 0,
	\]
	where the last equality is since $f$ is a homeomorphism from $\D$ to $\Y$ and $y \notin \Y$. We conclude that $\deg(h, U_i, y) = 0$ for every $i$. Since $J_h \geq 0$, the Jacobian formula of Lemma \ref{lem:degree_in_preimage} implies that $J_h \equiv 0$ on every $U_i$. By our assumption, this is impossible for a non-empty $U_i$; hence, we have reached a contradiction, proving that $h^{-1}\{y\} = \emptyset$.
	
	It remains to consider the case where $y \in \Y$, which proceeds similarly. We again let $U_1, U_2, \dots$ be the connected components of $h^{-1} B^2(y, r)$, and similarly as before we get $\deg(h, U_i, y) \geq 0$ and 
	\[
		\sum_i \deg(h, y, U_i) = \deg(f, y, \D \cap f^{-1}B^2(y, r)) = \pm 1,
	\]
	since $B^2(y, r) \subset \Y$ and $f$ is a homeomorphism from $\D$ to $\Y$. Since the degrees on the left hand side are non-negative, the value of $\deg(f, y, \D \cap f^{-1}B^2(y, r))$ must then be 1. It follows that there's exactly one component, which we may assume to be $U_1$, with $\deg(h, y, U_1) = 1$,. For any other components $U_i$, we have $\deg(h, y, U_i) = 0$.
	
	Hence, by e.g.\ \cite[Theorem 2.1]{Fonseca-Gangbo-book}, we must have $y \in f(U_1) \subset f(\D)$. This completes the proof that $\Y \subset f(\D)$. Moreover, by our assumption that $J_f \equiv 0$ on no non-empty component of $f^{-1}B^2(y, r)$, we cannot have any components $U_i$ with $\deg(h, y, U_i) = 0$. Hence, $U_1$ is the only component of $f^{-1}B^2(y, r)$. Since this holds for arbitrarily small $r$, and since every component of $h^{-1}\{y\}$ is compact by Corollary \ref{cor:capacity_to_topology_preimage}, it follows from Lemma \ref{lem:preimage_lemma_generalized} part \eqref{enum:component_pair} that $h^{-1}\{y\}$ must be connected.
\end{proof}

Lemma \ref{lem:degree_lemma_interior} hence yields connectedness of $h^{-1}\{y\} \cup (h\vert\S^1)^{-1}\{y\}$ for $y \in \Y$, as for such $y$ the boundary part of the pre-image is empty. Next, we consider pre-images of points $y \in \partial\Y$. We reduce the connectedness of $h^{-1}\{y\} \cup (h\vert\S^1)^{-1}\{y\}$ in this case to the following lemma.

\begin{lemma}\label{lem:preimage_component_boundary}
	Let $h \in C(\D, \C) \cap W^{1,2}(\D, \C)$ be such that $J_h \geq 0$ a.e.\ in $\D$ and $h$ satisfies the Lusin (N) condition. Suppose that $h - f \in W^{1,2}_0(\D)$, where $f \colon \overline{\D} \to \overline{\Y}$ is an orientation preserving homeomorphism onto a Lipschitz Jordan domain $\Y$. Suppose also that there exists no non-empty component $U$ of the pre-image $h^{-1} B$ of an open ball $B$ such that $J_f \equiv 0$ a.e.\ in $U$. If $y \in \partial\Y$, $z_0$ is the unique point for which $(h \vert \S^1)(z_0) = y$, and $K$ is a connected component of $h^{-1}\{y\}$, then $z_0 \in \overline{K}$.
\end{lemma}
\begin{proof}
	The proof divides into two cases. We consider first the case where $\overline{K} \cap \S^1 \neq \emptyset$. Now, Lemma \ref{lem:capacity_to_topology_lemma} applies, and we obtain that $f(\overline{K} \cap \S^1) \cap \overline{h(K)} \neq \emptyset$. Since $K$ is a component of $h^{-1}\{y\}$, we have $\overline{h(K)} = \{y\}$. Hence, $y \in f(\overline{K} \cap \S^1)$, and therefore $z_0 \in \overline{K} \cap \S^1 \subset \overline{K}$.
	
	It remains to consider the case where $\overline{K} \cap \S^1 = \emptyset$. We suppose towards contradiction that $z_0 \notin \overline{K}$. By Lemma \ref{lem:preimage_lemma_generalized} part \eqref{enum:component_boundary}, we find an $r > 0$ such that the connected component $U$ of $h^{-1}B^2(y,r)$ containing $K$ is compactly contained in $\D$.
	
	Now, $h$ is continuous on $\overline{U}$, and since $U$ is a connected component of $h^{-1}B^2(y,r)$, we have $h(\partial U) \cap B^2(y,r) = \emptyset$. It follows that the classical Brouwer degree $\deg(h, y', U)$ is well defined for every $y' \in B^2(y,r)$, including $y$. Moreover, since $B^2(y,r)$ is connected and $h(\partial U)$ is mapped outside it, $\deg(h, y', U)$ is independent of $y' \in B^2(y,r)$ by e.g.\ \cite[Theorem 2.3 (3)]{Fonseca-Gangbo-book}. Moreover, $h(\partial U) \subset \partial B^2(y, r)$ by continuity of $f$, so we again have by \cite[Proposition 5.25 and Remark 5.26 (ii)]{Fonseca-Gangbo-book} the Jacobian formula
	\[
	\int_U J_f(z) \dd z = \int_\C \deg(h, z', U) \dd z' = m_2(B^2(y,r)) \deg(h, y', U)
	\]
	for every $y' \in B^2(y,r)$.
	
	By our assumption that we cannot have $J_f \equiv 0$ in a component of the pre-image of a ball, we must have that $J_h > 0$ in a positive measured subset of $U$. The Jacobian formula hence implies that the constant value of $\deg(h, y', U)$ for $y' \in B^2(y,r)$ must be positive. Notably, every $y' \in B^2(y,r)$ must be in the image $h(U)$ by e.g.\ \cite[Theorem 2.1]{Fonseca-Gangbo-book}. This is a contradiction; $y$ is on the boundary of a Lipschitz Jordan domain $\Y$, so $B^2(y,r)$ must necessarily meet $\C \setminus \overline{\Y}$, yet $h(\D) \subset \overline{\Y}$ by Lemma \ref{lem:degree_lemma_interior}. Hence, our counterassumption has resulted in a contradiction, and the claim therefore holds.
\end{proof}

As previously indicated, the statement of Lemma \ref{lem:preimage_component_boundary} implies that the set  $h^{-1}\{y\}\cup(h\vert\S^1)^{-1}\{y\}$ is connected for every $y \in \partial \Y$. 

\begin{cor}\label{cor:monotone_for_points_mapped_into_bdry}
	Let $h \in C(\D, \C) \cap W^{1,2}(\D, \C)$ be such that $J_h \geq 0$ a.e.\ in $\D$ and $h$ satisfies the Lusin (N) condition. Suppose that $h - f \in W^{1,2}_0(\D)$, where $f \colon \overline{\D} \to \overline{\Y}$ is an orientation preserving homeomorphism onto a Lipschitz Jordan domain $\Y$. Suppose also that there exists no non-empty component $U$ of the pre-image $h^{-1} B$ of an open ball $B$ such that $J_f \equiv 0$ a.e.\ in $U$. Then for every boundary point $y \in \partial \Y$, the set $h^{-1}\{y\} \cup (h\vert\S^1)^{-1}\{y\}$ is connected.
\end{cor}
\begin{proof}
	Note that $(h\vert\S^1)^{-1}\{y\}$ is a singleton; we denote the single point in it by $z_0$. Suppose then that $(U, V)$ is a separation of $h^{-1}\{y\} \cup (h\vert\S^1)^{-1}\{y\}$. Either $U$ or $V$ must contain $(h\vert\S^1)^{-1}\{y\} = \{z_0\}$; we may assume by symmetry that $z_0 \in U$. 
	
	Let then $K$ be a connected component of $h^{-1}\{y\}$. Since $U$ is the intersection of $h^{-1}\{y\} \cup (h\vert\S^1)^{-1}\{y\}$ with an open subset of $\C$, and since Lemma \ref{lem:preimage_component_boundary} yields that $z_0 \in \overline{K}$, we must have that $U \cap K \neq \emptyset$. Since $K$ is connected, we must therefore have $V \cap K = \emptyset$. As this holds for all components $K$, we conclude that $V = \emptyset$, and therefore $h^{-1}\{y\} \cup (h\vert\S^1)^{-1}\{y\}$ is connected.
\end{proof}

\subsection{Applying the trajectory structure}

With Lemma \ref{lem:degree_lemma_interior} and Corollary \ref{cor:monotone_for_points_mapped_into_bdry} shown, the remaining key step is to eliminate the extra assumption that was required by their proofs. Namely, we wish to show that when $h$ is a solution of the Hopf-Laplace equation with non-negative Jacobian and a homeomorphic trace, then no pre-image $h^{-1}B$ of a ball has a non-empty component $U$ such that $J_h \equiv 0$ a.e.\ in $U$. 

\begin{lemma}\label{lem:zero_Jacobian_trajectory_lemma}
	Let $h \in W^{1,2}(\D, \C)$ be a solution to the Hopf-Laplace equation with $J_h \geq 0$ a.e.\ in $\D$.  Suppose that $h - f \in W^{1,2}_0(\D)$, where $f \colon \overline{\D} \to \overline{\Y}$ is an orientation preserving homeomorphism onto a Lipschitz Jordan domain $\Y$. Then for any $y \in \C$ and $r \in (0, \infty)$, there exists no non-empty connected component $U$ of $h^{-1} B^2(y, r)$ such that $J_h \equiv 0$ a.e.\ in $U$.
\end{lemma}
\begin{proof}
	By Theorem \ref{thm:hopf_harm_lipschitz}, $h$ is continuous inside $\D$ and satisfies the Lusin (N) -condition. We let $\varphi = (\partial_z h)(\overline{\partial_{\overline{z}} h})$ be the Hopf product of $h$, and let $H = \varphi dz \otimes dz$ be the associated holomorphic quadratic differential. We suppose to the contrary that $U$ is a connected component of $h^{-1} B^2(y, r)$ such that $J_f \equiv 0$ a.e.\ in $U$. Note again that $U$ is an open subset of $\D$.
	
	We consider first the simple case $H \equiv 0$. We then have $(\partial_z h) ( \overline{\partial_{\overline{z}} h}) = 0$ almost everywhere and therefore $\partial_z h(z) = 0$ or $\partial_{\overline{z}} h(z) = 0$ at a.e.\ $z \in D$. We also note that $\abs{\partial_z h}^2 - \abs{\partial_{\overline{z}} h}^2 = J_h \geq 0$ almost everywhere, which limits us to the possibility $\partial_{\overline{z}} h(z) = 0$ at a.e.\ $z \in D$. It follows that $h$ is weakly holomorphic, and therefore holomorphic by Weyl's lemma. We also know that $h$ is not constant, since it coincides with a homeomorphism on $\partial D$. It follows then that $J_h > 0$ outside of countably many isolated points, which contradicts our counterassumption that $J_h \equiv 0$ a.e.\ in $U$.
	
	Consider then the remaining case where $H$ is not identically zero.  Then, since $H$ only has isolated zeroes, the union of all critical vertical trajectories of $F$ has zero measure, as previously discussed in Section \ref{subsect:hopf_harm_prelims}. Hence, we can find a non-critical vertical trajectory $\gamma \colon (a, b) \to \D$ of $H$ which intersects $U$. Since $J_f \equiv 0$ a.e.\ in $U$, we have by Lemma \ref{lem:Hopf_harm_vert_traj_constant} that $f$ is constant on every segment of $\gamma$ contained in $U$.
	
	We then show that in fact the entire trajectory $\gamma$ is contained in $U$. Indeed, let $t \in (a,b)$ be such that $\gamma(t) \in U$, and suppose to the contrary that $\gamma$ also meets $\D \setminus U$. We may assume that $\gamma$ meets $\D \setminus U$ on $(t,b)$, as we can reverse the direction of $\gamma$ if necessary. Then there exists a smallest $t' > t$ such that $\gamma(t') \notin U$. By continuity of $\gamma$, we must hence have $\gamma(t') \in \partial U \cap \D$. It follows that $h$ is constant on $\gamma[t, t')$. Since $\gamma(t') \subset \D$ and $h$ is continuous in $\D$, we must then in fact have that $h(\gamma(t)) = h(\gamma(t'))$. It follows that $\gamma(t') \in h^{-1} B^2(y, r)$. But this is a contradiction, since $U$ is an open component of $h^{-1} B^2(y, r)$, yet $\gamma(t') \in \partial U$. Hence, $\gamma$ cannot meet $\D \setminus U$.
	
	It then follows that $h$ is constant on the entire vertical trajectory $\gamma$. Let $y_\gamma$ denote this constant value of $h$ on $\gamma$. Since $h \in W^{1,2}(\D, \C)$, we have $\varphi \in L^1(\D, \C)$ by H\"older's inequality. Hence, by Lemma \ref{lem:noncrit_traj_boundary_lemma}, the trajectory $\gamma$ tends to two distinct boundary points $z_a, z_b \in \S^1$ at $a$ and $b$ respectively. 
	
	We then split the image of $\gamma$ into two connected halves $C_a, C_b$ where $\overline{C_a} \cap \S^1 = \{z_a\}$ and $\overline{C_b} \cap \S^1 = \{z_b\}$. Now, Lemma \ref{lem:capacity_to_topology_lemma} yields that $\{f(z_a)\} \cap \{y_\gamma\} = f(\overline{C_a} \cap \S^1) \cap \overline{h(C_a)} \neq \emptyset$, and similarly $\{f(z_b)\} \cap \{y_\gamma\} = f(\overline{C_b} \cap \S^1) \cap \overline{h(C_b)} \neq \emptyset$. It follows that $f(z_a) = f(z_b) = y_\gamma$. But this is a contradiction, since $z_a \neq z_b$ and $f$ is a homeomorphism. Therefore, no set $U$ as above can exist.
\end{proof}

By combining Lemmas \ref{lem:degree_lemma_interior} and \ref{lem:zero_Jacobian_trajectory_lemma} and Corollary \ref{cor:monotone_for_points_mapped_into_bdry}, we immediately obtain the desired result of this section.

\begin{cor}\label{cor:monotone_for_points_mapped_into_Y}
	Let $h \in W^{1,2}(\D, \C)$ be a solution to the Hopf-Laplace equation with $J_h \geq 0$ a.e.\ in $\D$.  Suppose that $h - f \in W^{1,2}_0(\D)$, where $f \colon \overline{\D} \to \overline{\Y}$ is an orientation preserving homeomorphism onto a Lipschitz Jordan domain $\Y$. Then $\Y \subset h(\D) \subset \overline{\Y}$, $h^{-1}\{y\}$ is connected for every $y \in \Y$, and $h^{-1}\{y\} \cup (h\vert\S^1)^{-1}\{y\}$ is connected for every $y \in \partial \Y$.
\end{cor}

\section{Continuity up to the boundary}\label{Section_cont}

The last remaining part of the proof is to show that our solution is continuous up to the boundary. At this part we no longer require the Hopf-Laplace equation, so we formulate the statement for a more general class of mappings.

\begin{prop}\label{prop:general_boundary_extension}
	Let $\Y$ be a Lipschitz Jordan domain in $\C$, and let $h \colon \D \to \Y$ be a continuous surjection. Suppose that $h \in W^{1,2}(\D, \C)$, and that $h - f \in W^{1,2}_0(\D, \C)$ for some homeomorphic $f \colon \overline{\D} \to \overline{\Y}$. Suppose also that the set $h^{-1}\{y\} \cup (f\vert\partial\D)^{-1} \{y\}$ is connected for every $y \in \overline{\Y}$. Then extending $h$ to the boundary via $h\vert\partial\D = f\vert\partial\D$ yields a continuous map $h \colon \overline{\D} \to \overline{\Y}$.
\end{prop}

The claim is not true without the assumption that the pre-images of points are connected; see Example \ref{ex:2}. The argument is inspired by a proof that  a Sobolev homeomorphism $f \in W^{1,2}(\D, \Y)$ with Lipschitz $\Y$ extends continuously to the boundary, see \cite[Theorem 1.3]{Iwaniec-Onninen_boundary-behavior}. 

We begin with a lemma on pre-images of connected open sets. Recall that if $X$ and $Y$ are compact metric spaces and $f \colon X \to Y$ is a continuous monotone surjection, then a classical result of Whyburn states that $f^{-1}C$ is connected for every connected $C \subset Y$. For a proof, see e.g.\ \cite[Corollary 6.1.19]{Engelking_topology}. The compactness assumption is crucial; consider for example the map $h \colon (0,2)\times(0,2\pi + 1) \to \D$ defined by $h(x + iy) = \max(0, \abs{x} - 1) e^{i\min(y, 2\pi)}$ and $C = B^2(1/2, 1/5)$. Here, we require a version of this result for maps with Sobolev boundary values.

\begin{lemma}\label{lem:connectedness_of_monotone_preimages}
	Let $\Y$ be a Jordan domain in $\C$, and let $h \colon \D \to \Y$ be a continuous surjection. Suppose that $h \in W^{1,2}(\D, \C)$, and that $h - f \in W^{1,2}_0(\D, \C)$ for some homeomorphic $f \colon \overline{\D} \to \overline{\Y}$. Suppose also that the set $h^{-1}\{y\}$ is connected for every $y \in \Y$. Then for every connected open set $U \subset \Y$, the set $h^{-1} U$ is connected.
\end{lemma}
\begin{proof}[Proof of Lemma \ref{lem:connectedness_of_monotone_preimages}]
	We consider first the special case where $U = B^2(y, r_0)$ is a ball with $\overline{U} \subset \Y$. Let $V_i$ denote the connected components of $h^{-1} U$; the continuity of $h$ implies that the sets $V_i$ are open subsets of $\D$, which in turn implies that there are only countably many $V_i$. Moreover, by Corollary \ref{cor:capacity_to_topology_preimage} and the fact that $f(\partial\D) \subset \partial\Y$, we have $\overline{V_i} \subset \D$.
	
	Let $K = \overline{B^2}(y, r)$ for some $r < r_0$. We define $W_i = h^{-1} K \cap V_i$ for every index $i$. We claim that every $W_i$ is compact. Indeed, we have $\overline{W_i} \subset \overline{V_i} \subset \D$, and therefore the continuity of $h$ on $\D$ implies that $h(\overline{W_i}) \subset \overline{h(W_i)} \subset K$. On the other hand, since $V_i$ is a component of $h^{-1} U$ with $\overline{V_i} \subset \D$, we have $h(\partial V_i) \subset \partial U \subset \C \setminus K$. Hence, $\overline{W_i} \cap \partial V_i = \emptyset$, and therefore $\overline{W_i} \subset V_i$. Since $W_i$ is a closed subset of $V_i$, it follows that $W_i = \overline{W_i}$, implying the compactness of $W_i$.
	
	It follows that the sets $h(W_i)$ are compact subsets of $K$. Since $h$ is surjective, it follows that the sets $h(W_i)$ cover $K$. We also have that $h(W_i)$ are pairwise disjoint, since every fiber $h^{-1}\{w\}$ with $w \in K$ is connected and therefore contained in a single $V_i$. We hence have a cover of the compact continuum $K \subset \C$ with a countable collection of pairwise disjoint compact sets $h(W_i)$. A theorem of Sierpi\'nski now implies that no more than one $h(W_i)$ can be non-empty; see e.g.\ \cite[Theorem 6.1.27.]{Engelking_topology}. It follows that $h^{-1} K$ intersects only one component $V_i$ of $h^{-1} U$.
	
	Now, we select an increasing sequence of radii $0 < r_1 < r_2 < \dots$ with $\lim_{j \to \infty} r_j = r_0$, and denote $K_j = \overline{B^2}(y, r_j)$. Our proof so far shows that every $h^{-1} K_j$ is contained in a single component of $h^{-1} U$. Since the sequence of $h^{-1} K_j$ is increasing, this component is the same for all $j$. However, since $h^{-1} U = \bigcup_j h^{-1} K_j$, it then follows that $h^{-1} U$ must have exactly one connected component.
	
	It then remains to consider the case when $U$ is not a ball that is compactly contained in $\Y$. Let again $V_i$ be the connected components of $h^{-1} U$. Using the surjectivity and the fact that every pre-image $h^{-1}\{z\}$ for $z \in U$ is connected, it follows similarly to before that $h(V_i)$ are disjoint and cover $U$. Suppose towards contradiction that there exists more than one $h(V_i)$. Then we must have $\partial h(V_1) \cap \partial(U \setminus h(V_1)) \cap U \neq \emptyset$; otherwise $\overline{h(V_1)} \cap U$ and $\overline{U \setminus h(V_1)} \cap U$ would yield a separation of $U$.
	
	Let thus $z \in \partial h(V_1) \cap \partial(U \setminus h(V_1)) \cap U$, and let $B$ be a ball centered at $z$ with $\overline{B} \subset U$.  Then $B$ meets both $h(V_1)$ and $U \setminus h(V_1)$. Since every connected component of $h^{-1} B$ is contained in a connected component of $h^{-1} U$, it follows that $h^{-1} B$ must have at least 2 components; at least one contained in $V_1$, and at least one not contained in $V_1$. This contradicts the previous case, as $\overline{B} \subset U$.
\end{proof}

The basic idea of the proof of Proposition \ref{prop:general_boundary_extension} is to apply Lemma~\ref{lem:oscwm}. For that, we extend $h \colon \D \to \Y$ to a Sobolev map $H\in W_{\loc}^{1,2} (\C, \C)$ which takes $\C \setminus \D$ homeomorphically onto $\C \setminus \Y$. In order to show that the extended map $H$ satisfies the $K$-oscillation property also up to the boundary of $\D$, we use the following lemma.

\begin{lemma}\label{lem:weak_monotonicity}
	Let $\Y$ be a Lipschitz Jordan domain in $\C$, and let $h \colon \D \to \Y$ be a continuous surjection. Suppose that $h \in W^{1,2}(\D, \C)$, and that $h - f \in W^{1,2}_0(\D, \C)$ for some homeomorphic $f \colon \overline{\D} \to \overline{\Y}$. Suppose also that the set $h^{-1}\{y\} \cup (f\vert\S^1)^{-1} \{y\}$ is connected for every $y \in \overline{\Y}$.
	
	Extend $h$ to $\overline{\D}$ by $h\vert\S^1 = f\vert\S^1$ Then, for every $z_0 \in \S^1$ and for almost every $r \in (0, 1/2)$, $h$ is continuous on the boundary of $D_r := B^2(z_0, r) \cap \D$, and
	\begin{equation}\label{eq:osc_estimate}
		\sup_{\omega \in D_r} |h(\omega) - z| \leq \sup_{\omega \in \partial D_r} |h(\omega) - z| \qquad \forall z \in \C.
	\end{equation}
\end{lemma}
\begin{proof}
	The set $D_r$ is a Jordan domain whose boundary is the union of a circular arc inside $\D$ and a circular arc on $\partial \D$. Let $a,b \in \partial \D$ denote the intersection points of $\partial B^2(z_0,r)$ and $\partial \D$. 

	We then show that the map $h$ satisfies
	\begin{equation}\label{eq:circular_limits}
		\lim_{\substack{z \to a \\ z \in \partial B^2(z_0,r) \cap \D}} h(z) = f(a) \quad \text{ and } \quad \lim_{\substack{z \to b \\ z \in \partial B^2(z_0,r) \cap \D}} h(z) = f(b)
	\end{equation}
	for almost every $r \in (0, 1/2)$. Since $h$ is continuous inside $\D$ and $h\vert\S^1 = f\vert\S^1$ is continuous on $\S^1$, this will imply that $h$ is continuous on $\partial D_r$ for such $r$. In order to show \eqref{eq:circular_limits}, we select a homeomorphic extension $F: \C \to \C$ of $f$ such that $F \in W^{1,2}_{\loc}(\C)$, which we may do as discussed in Remark \ref{rem:Sobolev_extension}. We now define a map $H: \C \to \C$ by
	\[
	H(z) = \begin{cases} 
		h(z) & z\in \D \\
		F(z) & z \in \C \setminus \D
	\end{cases}\quad .
	\]
	Since $h-f \in W^{1,2}_0(\D, \C)$, the map $H$ is in $W^{1,2}_\loc(\C, \C)$. It follows that, after redefining $H$ in a set $S$ of measure zero, $H$ is absolutely continuous on $\partial B^2(z_0,r)$ for a.e.\ $r \in (0, 1/2)$.
	
	We also have for a.e.\ $r \in (0, 1/2)$ that $S \cap \partial B^{2}(x, r)$ has zero 1-dimensional measure. In such a case, if $H$ is absolutely continuous on $\partial B^2(z_0,r)$, then we must in fact have $S \cap \partial B^{2}(x, r) = \emptyset$. Indeed, if $H$ was redefined on $w \in S \cap \partial B^{2}(x, r) \cap (\C \setminus \D)$, then since $F$ is continuous, $H$ can be continuous on $\partial B^{2}(x, r)$ only if $S$ also contains some neighborhood of $w$ in $\partial B^{2}(x, r) \cap (\C \setminus \D)$. The same holds for $w \in \D$ due to continuity of $f$. Consequently, our original $H$ is continuous on $\partial B^2(z_0,r)$ for almost every $r \in (0, 1/2)$, which implies that \eqref{eq:circular_limits} holds.
	
	It remains then to prove the oscillation estimate \eqref{eq:osc_estimate}. To this end, let $\Gamma = h(\partial D_r)$ denote the image curve of $\partial D_r$ so that $\Gamma \subset \overline{\Y}$. The curve $\Gamma$ is compact and divides $\Y$ into a (possibly infinite) number of open connected components whose collection we denote by $\mathcal{Y}$. We now prove the following.
	\\\\
	\textbf{Claim.} Let $S \in \mathcal{Y}$ be a connected component of $\Y \setminus \Gamma$ and $\omega_0 \in D_r$. Suppose that $h(\omega_0) \in S$. Then $\partial S \subset \Gamma$.\\\\
	\emph{Proof of claim.} Suppose to the contrary that there exists a point $\tau \in \partial S$ with $\tau \notin \Gamma$. Since $S$ is a connected component of $\Y \setminus \Gamma$ we have $\partial S \subset \Gamma \cup \partial \Y$ and thus we must have $\tau \in \partial \Y$. Note now that since $S$ is an open and connected subset of $\Y$, Lemma \ref{lem:connectedness_of_monotone_preimages} yields that $h^{-1}(S)$ is also open and connected. Furthermore, $h^{-1}(S)$ does not intersect $\partial D_r$ by the definition of $\mathcal{Y}$ and thus $h^{-1}(S)$ must entirely lie either in $D_r$ or $\D \setminus \overline{D_r}$. But since $h(\omega_0) \in S$ where $\omega_0 \in D_r$ we must have that $h^{-1}(S) \subset D_r$.
	
	Since $\tau \notin \Gamma$ and since $\Gamma$ is compact, we may choose $\epsilon > 0$ such that $B^2(\tau,2\epsilon) \cap \Gamma = \emptyset$. Let $S_\epsilon$ be a connected component of $\Y \cap B^2(\tau,\epsilon)$ for which $\tau \in \overline{S_\epsilon}$. The set $S_\epsilon$ is open as it is a connected component of an open set. Moreover, since $B^2(\tau,\epsilon) \cap \Gamma = \emptyset$ we must have that $S_\epsilon$ is contained in a single connected component of $\Y \setminus \Gamma$ and because $\tau \in \partial S$ we must also have that $S_\epsilon \subset S$. By Lemma \ref{lem:connectedness_of_monotone_preimages}, the preimage of $S_\epsilon$ under $h$ is now an open and connected subset of $D_r$. We now consider two cases.
	\\\\
	\emph{Case 1.} If $\overline{h^{-1}(S_\epsilon)} \subset \D$.\\\\
	In this case, pick any sequence $(y_n) \subset S_\epsilon$ for which $\lim_{n\to\infty} y_n = \tau$. By the surjectivity of $h$, we may select a sequence of points $(x_n)$ with $x_n \in h^{-1}(y_n)$. Then $x_n \in h^{-1}(S_\epsilon)$ for each $n$ and we may choose a subsequence of $(x_n)$ which converges to a point $x \in \overline{h^{-1}(S_\epsilon)}$. Moreover, due to the choice of sequence $(y_n)$ and the continuity of $h$ inside $\D$ we must have that $h(x) = \tau$. But $\tau$ must also be the image of some point $x' \in \partial \D$ under the boundary map $f \vert \S^1 = h \vert \S^1$. Since the preimage $h^{-1}(\tau)$ of $\tau$ under $h$ must be a connected set by our assumptions, we find that $h^{-1}(\tau)$ must intersect $\partial D_r$, as otherwise the sets $D_r$ and $\C \setminus D_r$ would yield a separation of it. But this is a contradiction as now $\tau \in h(\partial D_r) = \Gamma$.\\\\
	\emph{Case 2.} If $\overline{h^{-1}(S_\epsilon)} \cap \partial \D \neq \emptyset$.\\\\
	Now $h^{-1}(S_\epsilon)$ is a connected subset of $\D$ whose closure meets $\S^1$. It follows by Lemma \ref{lem:capacity_to_topology_lemma} that $f(\overline{h^{-1}(S_\epsilon)}\cap\S^1) \cap \overline{h(h^{-1}S_\epsilon))} \neq \emptyset$. However, since $S_\epsilon \subset D_r$, we have $f(\overline{h^{-1}(S_\epsilon)}\cap\S^1) \subset f(\partial D_r \cap \S^1) \subset \Gamma$, and since $S_\epsilon$ is a component of $B^2(\tau,\epsilon)$, we have $\overline{h(h^{-1}S_\epsilon))} \subset \overline{S_\epsilon} \subset B^2(\tau, 2\epsilon)$. Since we chose $\epsilon$ so that $B^2(\tau, 2\epsilon) \cap \Gamma = \emptyset$, we have reached a contradiction.	
	\\\\
	Returning to the proof of \eqref{eq:osc_estimate}, we have now shown that if any point $\omega_0 \in D_r$ is mapped inside $\Y$ then it is contained in an open and connected set whose boundary is a subset of $h(\partial D_r)$. This means that for such $\omega_0$ the quantity $|h(\omega_0) - z|$ for fixed $z \in \C$ may always be increased by replacing $\omega_0$ with some point on $\partial D_r$.
	
	It remains to consider the case when $h(\omega_0) \in \partial \Y$. But in this case we can repeat the final argument of Case 1 above: The preimage of $h(\omega_0)$ under $h$ must be a connected set, and since it contains both $\omega_0$ and a point on $\partial\D$, it must also intersect $\partial D_r$. This proves \eqref{eq:osc_estimate}.
\end{proof}

It then remains to complete the proof of Proposition \ref{prop:general_boundary_extension}.

\begin{proof}[Proof of Proposition \ref{prop:general_boundary_extension}]
	Again, as discussed in Remark \ref{rem:Sobolev_extension}, we can select a homeomorphic extension $F: \C \to \C$ of $f$ such that $F \in W^{1,2}_{\loc}(\C, \C)$, and then define a map $H \in W^{1,2}_\loc(\C, \C)$ by
	\[
	H(z) = \begin{cases} 
		h(z) & z\in \D \\
		F(z) & z \in \C \setminus \D
	\end{cases}\quad .
	\]
	Let $z \in \partial \D$ and $r \in (0, 1/2)$. We wish to show that for almost every such $r$ the estimate
	\[\underset{B^2(z,r)}{\osc} H \leq C\underset{\partial B^2(z,r)}{\osc} H\]
	is valid, where $C$ does not depend on $z$ or $r$. Let $D_r = B^2(z,r) \cap \D$ and $D_r' = B^2(z,r) \setminus \overline{\D}$. Now, the estimate \eqref{eq:osc_estimate} of Lemma \ref{lem:weak_monotonicity} is valid for almost all of our $r$, and implies for such $r$ that
	\[\sup_{\omega \in D_r} |H(\omega) - z| \leq \sup_{\omega \in \partial D_r} |H(\omega) - z| \qquad \forall z \in \C.\]
	Furthermore, since $H$ is a homeomorphism on $D_r'$ we find analogously that
	\[\sup_{\omega \in D_r'} |H(\omega) - z| \leq \sup_{\omega \in \partial D_r'} |H(\omega) - z| \qquad \forall z \in \C.\]
	Combining the above two estimates yields
	\begin{equation}\label{eq:combined_osc_estim}
		\underset{B^2(z,r)}{\osc} H \leq \underset{\partial D_r \cup \partial D_r'}{\osc} H.
	\end{equation}
	Let $\Psi : \C \to \C$ be a bilipschitz map which takes the boundary curve $H(\partial \D \cap B^2(z,r))$ onto a line segment of equal length. Such a map exists because $\Y$ is a Lipschitz domain, and the bilipschitz constant $C_\Psi$ of $\Psi$ is controlled by a constant only depending on $\Y$. 
	
	Let $\omega_-$ and $\omega_+$ denote the two endpoints of the circular arc $\partial \D \cap B^2(z,r)$. Note now that since $\Psi(H(\gamma))$ is a line segment we have that
	\[\sup_{\omega \in \partial \D \cap B^2(z,r)}|\Psi(H(\omega)) - z| \leq \max(|\Psi(H(\omega_-)) - z|, |\Psi(H(\omega_+)) - z|)\qquad \forall z \in \C.\]
	This implies that the maximal oscillation of $\Psi(H(\omega))$ on the set $\partial D_r \cup \partial D_r'$, which is a union of $\partial B^2(z,r)$ and the circular arc between $\omega_-$ and $\omega_+$, is found on $\partial B^2(z,r)$. In other words, we may combine the above estimate with \eqref{eq:combined_osc_estim} to find that
	\begin{align*}
		\underset{B^2(z,r)}{\osc} H &\leq \underset{\partial D_r \cup \partial D_r'}{\osc} H
		\leq C_\Psi \underset{\partial D_r \cup \partial D_r'}{\osc} \Psi \circ H
		\\&\leq C_\Psi \underset{\partial B^2(z,r)}{\osc} \Psi \circ H
		\leq C_\Psi^2 \underset{\partial B^2(z,r)}{\osc} H.
	\end{align*}
	This is exactly what we want as $C_\Psi$ only depends on $\Y$. Thus, Proposition~\ref{prop:general_boundary_extension} follows from Lemma~\ref{lem:oscwm}.
\end{proof}

With the proof of Proposition \ref{prop:general_boundary_extension} complete, we have also completed the proof of Theorem \ref{thm:hopf_monotonicity_upgraded}, and consequently also the proof of our main result, Theorem \ref{thm:main}.

\section{Counterexamples}\label{sect:counterexamples}

In this section we provide justification to Example \ref{ex:1} and \ref{ex:2} by constructing the required maps.

We start by defining types of auxiliary maps which will be used in both constructions. The first map, denoted by $G_{\omega,s,a}$ for $\omega \in \C$ and $s,a > 0$, is defined as a map of the disk $\overline{B(\omega,2s)}$ to itself by the following formula
 \begin{equation*}
 G_{\omega,s,a}(\omega + r e^{i\theta}) = \begin{cases} \omega + 2(r-s)e^{i\theta} & \qquad  \textnormal{for } s \leq r \leq 2s \\
 \omega + \frac{a (s - r)}{s}   & \qquad  \textnormal{for } r < s
 \end{cases}.
 \end{equation*}
Thus, $G_{\omega,s,a}$ is the identity mapping on the boundary, maps the set $\overline{B(\omega,2s)}\setminus \overline{B(\omega,s)}$ homeomorphically into the punctured disk $\overline{B(\omega,2s)}\setminus \{\omega\}$, and maps the set $\overline{B(\omega,s)}$ onto the closed line segment from $\omega$ to $\omega + a$. This map notably also has a nonnegative Jacobian and is Lipschitz continuous.

For a given $\epsilon > 0$ satisfying $\epsilon < s$, we also define another variant $G_{\omega,s,a,\epsilon}$ of the above map by
 \begin{equation*}
	G_{\omega,s,a,\epsilon}(\omega + r e^{i\theta}) = \begin{cases} \omega + 2(r-s)e^{i\theta} & \qquad  \textnormal{for } s \leq r \leq 2s \\
		\omega + \frac{\log(s/r)}{\log(s/\epsilon)} a   & \qquad  \textnormal{for } \epsilon < r < s\\
		\omega + a  & \qquad  \textnormal{for } r \leq \epsilon
	\end{cases}.
\end{equation*}
This mapping is topologically similar to the previous one but maps the circles $S(\omega,r)$ for $r < s$ to the line segment logarithmically instead of linearly. Due to this, we may calculate that

\begin{align*}
	\int_{B(\omega,2s)} |D G_{\omega,s,a,\epsilon}(z)|^2 \, dz &\leq C_1 s^2 + 2\pi \int_{\epsilon}^s \frac{\abs{a}^2}{r^2 \log^{2} (s/\epsilon)} r dr
	\\
	&= C_1 s^2 + \frac{2\pi\abs{a}^2}{\log^{2} (s/\epsilon)} \int_{\epsilon}^s \frac{1}{r} dr\\
	&= C_1 s^2 + \frac{2\pi\abs{a}^2}{\log (s/\epsilon)}.
\end{align*}

\begin{proof}[Proof of Example \ref{ex:2}] We choose $G(0,1/2,1)$ as our map. This is a mapping of $\D$ to itself, it belongs to $\mathcal A^2_{\id} (\overline{\D}, \C)$, and the preimage of any point on $(0,1)$ is the union of a single point in $\D \setminus \overline{B(0,1/2)}$ and a circle with center $0$ and radius in $(0,1/2)$. Our claim follows.
\end{proof}

\begin{proof}[Proof of Example \ref{ex:1}] We start with the identity mapping on $\overline{\D}$. We then choose a sequence of disjoint disks $B(\omega_k,2s_k) \subset \D$ for $k = 1,2,\ldots$ which get gradually smaller and $\lim_{k \to \infty} \omega_k = 1$. We select for example $\omega_k = 1 - 1/k$ and $s_k = 10^{-k}$. Then on each $B(\omega_k,2s_k)$, we replace our identity mapping with the map $G_{\omega_k,s_k,2,\epsilon_k}$, where for example $\epsilon_k = e^{-k^4}s_k$; i.e.\ we define a map $G : \overline{\D} \to \overline{\D}$ by
 \begin{equation*}
 G(z) = \begin{cases} G_{\omega_k,s_k,2,k^{-1}}(z) & \qquad  \textnormal{for } z \in B(\omega_k,2s_k),\, k = 1,2,\ldots \\
 z   & \qquad  \textnormal{otherwise}
 \end{cases}.
 \end{equation*}
This mapping is continuous in $\D$ but cannot be continuous up to the boundary because it maps each of the centers $\omega_k$ to the point $\omega_k + 2$ with real part larger than $2$. Furthermore,
\begin{align*}
	\norm{G}_{W^{1,2}(\D,\C)} &\leq C_0 + \sum_{k=1}^\infty \norm{DG_{\omega_k,s_k,2,\epsilon_k}(z)}_{L^2(B(\omega_k,2s_k))} \\
	&\leq C_0 + \sum_{k=1}^\infty \sqrt{C_1 10^{-2k} + \frac{8\pi}{\log {e^{k^4}}}}
	\\&\leq C_0 + C_2\sum_{k=1}^\infty\left( 10^{-k} + \frac{1}{k^2}\right) < \infty.
\end{align*}
Thus, $G$ belongs to the Sobolev class $W^{1,2}(\D,\C)$. The trace of $G$ on the boundary must also equal the identity mapping, since e.g.\ $\lim_{z\to \omega} G(z) = \omega$ for all $\omega \in \partial \D \setminus \{1\}$. Hence, $G$ satisfies all the claimed properties.
\end{proof}



\end{document}